%% file: NewGameplan.tex
\documentclass[12pt,a4paper,reqno]{amsart}
\usepackage[left=2cm, right=2cm, bottom=3cm]{geometry} 
\usepackage{xcolor}

\newcommand\dela[1]{}

\usepackage[english]{babel}
\usepackage{palatino}

\usepackage{amsmath}
\usepackage{amsthm}
\usepackage{amssymb}
\usepackage{url}
\usepackage{hyperref}
\usepackage{graphicx}
\usepackage{url}
\usepackage{listings}
\usepackage{nicefrac}
\usepackage{tikz}
\usetikzlibrary{calc,fadings,decorations.pathreplacing}
\usepackage{standalone}
\usepackage{pgfplots}
\usepackage{pgfplotstable}
\usepgfplotslibrary{fillbetween}
\usetikzlibrary{patterns}
\makeindex
\usepackage{enumerate}

\usepackage[framemethod=tikz]{mdframed}
\mdfdefinestyle{fancy}{
	roundcorner=5pt,
	linewidth=4pt,
	linecolor=red!80,
	backgroundcolor=red!20
}

\allowdisplaybreaks

\newtheorem{Satz}{Theorem}[section]
\newtheorem{Theorem}[Satz]{Theorem}
\newtheorem{Lemma}[Satz]{Lemma}		
\newtheorem{Korollar}[Satz]{Corollary}	
\newtheorem{Prop}[Satz]{Proposition}	
\numberwithin{equation}{section}
\theoremstyle{definition}

\newtheorem{Definition}[Satz]{Definition}

\newtheorem{Assumption}[Satz]{Assumption}

\newcommand{\BIGboxplus}{\mathop{\mathchoice%
		{\raise-0.35em\hbox{\huge $\boxplus$}}%
		{\raise-0.15em\hbox{\Large $\boxplus$}}{\hbox{\large $\boxplus$}}{\boxplus}}}

\newcommand{\C}{\mathbb{C}} 
\newcommand{\K}{\mathbb{K}} 
\newcommand{\R}{\mathbb{R}} 
\newcommand{\Rd}{{\mathbb{R}^d}} 
\newcommand{\Z}{\mathbb{Z}} 
\newcommand{\N}{\mathbb{N}} 

\newcommand{\E}{\mathbb{E}}
\newcommand{\Prob}{\mathbb{P}}
\newcommand{\F}{\mathcal{F}}
\newcommand{\Filtration}{\mathbb{F}}
\newcommand{\tildeProb}{\tilde{\Prob}}
\newcommand{\Etilde}{\tilde{\E}}
\newcommand{\HS}{\operatorname{HS}}

\newcommand{\Yosida}{R_\lambda}

\newcommand{\energy}{\mathcal{E}}

\newcommand{\EA}{{E_A}}

\newcommand{\EAdual}{{E_A^*}}

\newcommand{\LinftyEA}{{L^\infty(0,T;\EA)}}

\newcommand{\LInfty}{{L^\infty(M)}}

\newcommand{\weaklyContinousEA}{C_w([0,T],\EA)}

\newcommand{\LalphaPlusEins}{{L^{\alpha+1}(M)}}

\newcommand{\LalphaPlusEinsDual}{{L^{\frac{\alpha+1}{\alpha}}(M)}}
\newcommand{\LalphaPlusEinsDualNorm}{{L^{(\alpha+1)/{\alpha}}}}

\newcommand{\LalphaPlusEinsKurz}{{L^{\alpha+1}}}

\newcommand{\D}{\mathcal{D}}
\newcommand{\df }{\mathrm{d}}
\newcommand{\im }{\mathrm{i}}
\newcommand{\sumM }{\sum_{m=1}^{\infty}}
\newcommand{\Real}{\operatorname{Re}}

\newcommand{\skpH}[2]{\big(#1,#2\big)_{\!H}}
\newcommand{\skp}[2]{\big(#1,#2\big)}
\newcommand{\skpHReal}[2]{\Real \big(#1,#2\big)_{H}}

\newcommand{\sqrtA}{A^{\frac{1}{2}}}

\newcommand{\norm}[1]{\Vert #1 \Vert}
\newcommand{\bigNorm}[1]{\left\Vert #1 \right\Vert}
\newcommand{\quadVar}[1]{\langle \langle #1 \rangle \rangle}
\newcommand{\duality}[2]{\langle #1, #2 \rangle}
\newcommand{\dualityReal}[2]{\Real \langle #1, #2 \rangle}

\newcommand{\Fhat}{\hat{F}}

\newcommand{\LinearOperators}[1]{{\mathcal{L}(#1)}}
\newcommand{\LinearOperatorsTwo}[2]{{\mathcal{L}(#1,#2)}}
\newcommand{\Id}{\operatorname{Id}}

\newcommand{\normEA}{{E_A}}
\newcommand{\normEAdual}{{E_A^*}}

\title{The stochastic nonlinear Schr\"odinger equation in unbounded domains and manifolds}

\author{Fabian Hornung$^1$}
\thanks{$^1$Institute for Analysis, Karlsruhe Institute of Technology, Englerstra\ss{}e 2, 76137 Karlsruhe, Germany.
	E-mail: fabianhornung89@gmail.com}

\date{\today}

\begin{document}

\begin{abstract}
	In this article, we construct a global martingale solution to  a general nonlinear Schr\"{o}dinger equation with linear multiplicative noise in the Stratonovich form. Our framework includes many examples of spatial domains like $\Rd$, non-compact Riemannian manifolds, and unbounded domains in $\Rd$ with different boundary conditions. The initial value belongs to the energy space $H^1$ and we treat subcritical focusing and defocusing  power nonlinearities. The proof is based on an approximation technique which makes use of spectral theoretic methods and an abstract Littlewood-Paley-decomposition. In the limit procedure, we employ tightness of the approximated solutions and Jakubowski's extension of the Skorohod Theorem to nonmetric spaces. 
\end{abstract}
\maketitle



\input{Introduction}
\input{Assumptions}
\input{Truncation}
\input{Compactness}

\input{Limit}
\input{Examples}
\section*{Acknowledgement}

The author gratefully acknowledges financial support by the
Deutsche Forschungsgemeinschaft (DFG) through CRC 1173. Moreover, he would like to thank Martin Spitz for helpful discussions on Section 6. 

\end{document}

%% file: Introduction.tex
\section{Introduction}
We study the existence of a solution to the nonlinear Schr\"odinger equation (NLS) with multiplicative Stratonovich noise 
\begin{equation}
\label{ProblemStratonovich}
\left\{
\begin{aligned}
\df u(t)&= \left(-\im A u(t)-\im  F(u(t))\right) dt-\im \sumM e_m u(t) \circ \df \beta_m(t),\\
u(0)&=u_0,
\end{aligned}\right.
\end{equation}
in the energy space $E_A:=\D(\sqrtA),$ where $A$ is a selfadjoint, non-negative operator $A$ in a Hilbert space $H=L^2(M),$ $F\colon \EA\to \EAdual$ is a defocusing or focusing power-type nonlinearity,  $e_m$, $m\in \N$,  are  real-valued functions,  and $\beta_m$, $m\in\N$, are independent Brownian motions. Let us immediately formulate the main result of this paper which contains the existence of a global martingale solution of \eqref{ProblemStratonovich} for various different choices of $A$, $M$, $F$, and $e_m$, $m\in\N$.

\begin{Theorem}\label{mainTheoremIntro}
	Suppose that a) or b) or c) or d) is true. 
	\begin{itemize}
		\item[a)] Let $M\subset \Rd$ be a  Lipschitz domain, $A=-\Delta_N$ be the Neumann-Laplacian and \\$E_A=H^1(M).$ 
		\item[b)] Let $M\subset \Rd$ be a  domain, $A=-\Delta_D$ be the Dirichlet-Laplacian and $E_A=H^1_0(M).$ 
				\item[c)] Let $M=\Rd$, $A=-\Delta$, and $\EA=H^1(\Rd)$.
		\item[d)] Let $(M,g)$ be a complete Riemannian manifold with positive injectivity radius,  bounded geometry, and nonnegative Ricci curvature; let $A=-\Delta_g$ be the Laplace-Beltrami operator, and $E_A=H^1(M).$
	\end{itemize}	
	Choose the nonlinearity from $i)$ or $ii).$
	\begin{itemize}
		\item[i)] $F(w)= \vert w\vert^{\alpha-1}w$ with $ \alpha \in \left(1,1+\frac{4}{(d-2)_+}\right),$
		\item[ii)]$F(w)= -\vert w\vert^{\alpha-1}w$ with $ \alpha \in \left(1,1+\frac{4}{d}\right).$
	\end{itemize}
	Assume $u_0\in\EA$ and that the coefficients $e_m\colon M\to\R$ satisfy $\sumM \norm{e_m}_E^2<\infty$ for 
	\begin{align}\label{MultiplierClassIntro}
	E:=\begin{cases}
	H^{1,d}(M) \cap \LInfty, &  d\ge 3,\\
	H^{1,q}(M),& d=2,\\
	H^{1}(M),& d=1,\\
	\end{cases}
	\end{align}
	for some $q>2$ in the case $d=2.$ 
	Then, \eqref{ProblemStratonovich} has a global  
	martingale solution  $\left({\Omega}',{\F}',{\Prob}',{W}',{\Filtration}',u\right)$ in $\EA$ which satisfies $u\in \weaklyContinousEA$ almost surely, $u\in L^q({\Omega}',\LinftyEA)$ 
	for all $q\in [1,\infty)$, and
	$\norm{u(t)}_{L^2(M)}=\norm{u_0}_{L^2(M)}$
	almost surely for all $t\in [0,T].$
\end{Theorem}

The present paper can be viewed as a continuation of the studies in the joint work of the author with Brzezniak and Weis in \cite{ExistencePaper} and the author's dissertation \cite{dissertationFH}. There, an existence theory for equation \eqref{ProblemStratonovich} was developed in general framework based on a variant of the Faedo-Galerkin method replacing some of the orthogonal projections in the approximating equation by means of Littlewood-Paley theory. Roughly speaking, the authors had to assume that $A$ is a Laplace-type operator in the sense that its heat semigroup has generalized Gaussian bounds, that $F$ is a subcritical nonlinearity, that the volume of $M$ is finite, and that certain Sobolev embeddings are compact. In contrast to the first two assumptions which are typical for nonlinear Schr\"odinger equations, the latter two assumptions are somehow unnatural and restrict the results in \cite{ExistencePaper} and \cite{dissertationFH} to bounded domains and compact manifolds. In particular, the most common case of $M=\Rd$ was not admissible. The reason for these restrictions is the fact that the authors needed a limit procedure in the strong topology for the nonlinear term. 

The goal of this study is a generalization of these results by dropping the assumption of finite volume and compactness of Sobolev embeddings. In view of the existing literature on the deterministic NLS (cf., e.g. \cite{Bolleyer, Cazenave}), we expect that in this paper, we derive the most general result in terms of spatial domains and differential operators that can be archieved for the stochastic NLS by approximation techniques. In Theorem \ref{mainTheoremIntro}, we only cover the most illustrative examples of the unified framework presented below. We point out that similarly to \cite{ExistencePaper}, one can also treat the fractional NLS on domains and manifolds and replacing the Laplacians by more general elliptic operators is also admissible.
Moreover, there are other somehow odd examples like the NLS on graphs or fractals. In order to keep the length of the paper reasonable, we concentrate on the examples mentioned in Theorem \ref{mainTheoremIntro}. For the general result, we refer to Theorem~\ref{MainTheorem} below. Let us sketch our argument which is based on weak limits and  mainly inspired by the deterministic results of \cite{ginibreVelo1985}, Chapter 3 in \cite{Cazenave}, and Theorem I.3.6 in \cite{Bolleyer}  and ideas from \cite{BrzezniakMotyl} for the stochastic Navier-Stokes equation.
Quite similarly to \cite{ExistencePaper}, \cite{Brzeźniak2019}, and \cite{dissertationFH} we truncate \eqref{ProblemStratonovich} by
\begin{equation}
\left\{
\begin{aligned}\label{GalerkinIntroduction}
\df u_n(t)&= \left(-\im A u_n(t)-\im P_n F\left( u_n(t)\right) \right) \df t-\im \sumM S_n[ e_m S_n u_n(t))] \circ \df W(t),\quad t>0,
\\
u_n(0)&=S_n u_0,
\end{aligned}\right.
\end{equation}
where $P_n$ is a spectral projection associated to $A$  and $S_n$
is a selfadjoint operators derived from the  Littlewood-Paley-decomposition associated to $A.$ The operators $P_n$ and $S_n$ are further illustrated in Figure 1.

\begin{figure}[h]
	\begin{minipage}{0.49\textwidth}

		\begin{tikzpicture}
		
		
		\draw[->] (-0.2,0) -- (5,0) node[anchor=north] {$\lambda$};
		\draw[->] (0,-0.2) -- (0,2) node[anchor=east] {$p_n(\lambda)$};
		
		
		\draw	(-0.2,0) node[anchor=east] {$0$}
		(-0.2,1.5) node[anchor=east] {$1$}
		(0,-0.2) node[anchor=north] {$0$}
		(2,-0.2) node[anchor=north] {$2^n$}
		(4,-0.2) node[anchor=north] {$2^{n+1}$};
		
		\draw (2,-0.1)--(2,0.1);
		\draw (4,-0.1)--(4,0.1);
		
		
		\draw [line width=1pt](0.014,1.5)--(4,1.5);
		\draw [ dashed,line width=0.5pt](4,1.5)--(4,0.002);
		\draw [line width=1pt](3.991,0)--(4.95,0);
		\end{tikzpicture}
	\end{minipage}
	\begin{minipage}{0.49\textwidth}

		\begin{tikzpicture}
		
		\draw[->] (-0.2,0) -- (5,0) node[anchor=north] {$\lambda$};
		\draw[->] (0,-0.2) -- (0,2) node[anchor=east] {$s_n(\lambda)$};

		
		\draw	(-0.2,0) node[anchor=east] {$0$}
		(-0.2,1.5) node[anchor=east] {$1$}
		(0,-0.2) node[anchor=north] {$0$}
		(2,-0.2) node[anchor=north] {$2^n$}
		(4,-0.2) node[anchor=north] {$2^{n+1}$};
		\draw (2,-0.1)--(2,0.1);
		\draw (4,-0.1)--(4,0.1);
		
		
		\draw [line width=1pt](0.014,1.5)--(2,1.5);
		\draw [line width=1pt](3.991,0)--(4.99,0);
		\draw[domain=2:4, variable=\x,line width=1pt] plot (\x,{1.5*(-6*(0.5*\x-1)^5+15*(0.5*\x-1)^4-10*(0.5*\x-1)^3+1)});
		
		\end{tikzpicture}
	\end{minipage}
	\caption{Functions $p_n$ and $s_n$ with $P_n=p_n(I+A)$ and $S_n=s_n(I+A).$}
\end{figure}
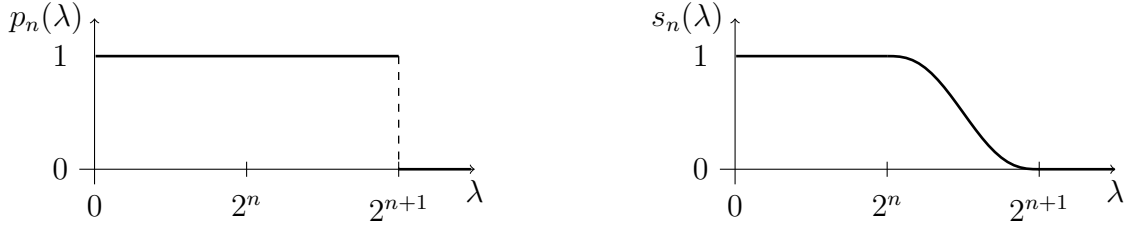

The reason for using the operators $\left(S_n\right)_{n\in\N_0}$ lies in the uniform estimate
\begin{align}\label{SnBoundedness}
\sup_{n\in\N_0}\norm{S_n}_{L^p\to L^p}<\infty, \quad 1<p<\infty,
\end{align}
needed to estimate the noise and which is false if one replaces $S_n$ by $P_n.$
 As a major difference to \cite{ExistencePaper} and \cite{dissertationFH}, observe that by the lack of compact embeddings, \eqref{GalerkinIntroduction} is not a finite dimensional stochastic differential equation but a stochastic evolution equation in the range of $P_n$ which is an infinite dimensional closed subspace of $H=L^2(M)$. However, the smoothing  properties of $P_n$ in the scale of fractional domains of $A$, i.e.
 		\begin{align}\label{smoothing}
 		\norm{P_n}_{\LinearOperatorsTwo{H}{\EA}}\le C_n \qquad\text{and}\qquad  \norm{P_n}_{\LinearOperatorsTwo{\EAdual}{H}}\le C_n
 		\end{align} 
 		for some constant $C_n\in (0,\infty)$ depending on $n\in\N_0$,
  and the symmetric structure of the noise truncation still guarantee that there is a unique global strong solution of \eqref{GalerkinIntroduction}. We remark that besides its application to the NLS, a similar construction has also  been employed in \cite{hornung2017strong} to construct a solution of a stochastic nonlinear Maxwell equation. This indicates that using operators like $S_n,$ $n\in\N_0$,  significantly increases the potential of the classical Faedo-Galerkin method.
  For further illustration, we point out that in the special case c), i.e. $M=\R^d$, $A=-\Delta$, the operators are Fourier multipliers as they can be written in the form
  \begin{align*}
  	P_n x(\xi)= \mathcal{F}^{-1} \big[p_n(\vert \cdot\vert^2) \,\mathcal{F}x\big](\xi)\qquad \text{and}\qquad S_n x(\xi)= \mathcal{F}^{-1} \big[s_n(\vert \cdot\vert^2) \,\mathcal{F}x\big](\xi)
  \end{align*}
  for $x\in L^2(\Rd)$ with the real valued functions $p_n$ and $s_n$ from Figure 1. 
  Here, the crucial estimate \eqref{SnBoundedness} is a consequence of the classical Mihlin Theorem on the boundedness of Fourier multipliers in $L^p(\Rd)$ and \eqref{smoothing} is a special case of the Bernstein inequality for functions with the Fourier transform supported in a ball. 
  
After this interlude, we would like to continue our sketch of the existence proof. A Gronwall type argument yields the uniform a priori estimates
\begin{align}\label{uniformIntroduction}
\sup_{n\in\N_0}\E \Big[\sup_{t\in[0,T]} \norm{u_n(t)}_\EA^q\Big]<\infty ,\qquad
\sup_{n\in\N_0}\E \Big[\sup_{t\in[0,T]} \norm{F(u_n(t))}_\LalphaPlusEinsDualNorm^{q}\Big]<\infty
\end{align}
for all $q\in [1,\infty)$.
Let us denote the scale of fractional domains of $A$ by $(X_s)_{s\in\N}$ and assume that that there is some $\gamma<1/2$ with $X_\gamma\hookrightarrow \LalphaPlusEins$. Note that in the applications, this assumption is typically true for subcritical nonlinearities and false for critical ones.  As in \cite{BrzezniakMotyl}, we take a Hilbert space $U^*$ which is compactly embedded in $\EA$. 
Based on \eqref{uniformIntroduction}, we then prove the tightness of the sequence $((u_n, F(u_n)))_{n\in\N_0}$ in the space 
\begin{align}
Z= C_w([0,T],\EA\times \LalphaPlusEinsDual)\cap C([0,T],X_\gamma\times \LalphaPlusEinsDual)
\end{align}
equipped with a topology inherited by strong convergence in $C([0,T],U^*\times U^*)$ and weak convergence in $\EA\times \LalphaPlusEinsDual$. 
 Jakubowski's extension of the Skorohod Theorem to nonmetric spaces from \cite{Jakubowski} yields a sequence of random variables $(v_k,\tilde{F}_k)_{k\in\N}$ on an enlarged probability space $\tilde{\Omega}$ which coincides in law with the original sequence and converges almost surely in $Z$ to $(v,\tilde{F})$. The function $v$ turns out to be a natural candidate for the global martingale solution since we prove 
 	\begin{align*}
 	v(t)= u_0+ \int_0^t \left[-\im A v(s)-\im  \tilde{F}(s)\right] \df s-\im \sumM\int_0^t e_m v(s)\circ\df \beta_m'(s)
 	\end{align*}  
for some sequence $(\beta_m')_{m\in\N}$ of independent Brownian motions on the large probability space. Finally, we prove $\tilde{F}=F(v)$ and thus, $v$ is indeed the martingale solution we were looking for.

We now classify Theorem \ref{mainTheoremIntro} into the literature on the analysis of the stochastic NLS which has catched significant attention since the first work \cite{BouardLzwei} of de Bouard and Debussche twenty years ago and particularly in the last five years. In contrast to the present article, most authors use fixed arguments based on Strichartz estimates which on the one hand, has the advantage that existence and uniqueness results can be proved simultaneously, but on the other hand, typically restricts the results to specific geometries. We refer to Barbu, R\"ockner, Zhang \cite{BarbuL2,BarbuH1} and Zhang \cite{zhang2017,zhangCritical} for wellposedness on $\Rd$ based on a rescaling transformation, de Bouard, Debussche \cite{BouardLzwei,BouardHeins},
 Fan, Xu \cite{fan2018global}, 
and the author \cite{FHornung} for existence and uniqueness on $\Rd$ via estimates for the stochastic convolution, and to Cheung, Mosincat \cite{cheung2019stochastic} and Brzezniak, Millet \cite{BrzezniakStrichartz} for similar results on the torus and general compact manifolds, respectively. 
We emphasize that our approach is tailormade for the general setting and we would like to admit that our procedure is therefore not the natural one to deal with $M=\Rd$ as we neither use  dispersive behavior of the Schr\"odinger group on $\Rd$ nor the full strength of harmonic analysis in this case. Despite this deficit, it is worth mentioning that even in the $\Rd$-setting, our existence results extends the literature. On the one hand, we reduce the regularity and decay assumptions on the noise coefficients in \cite{BarbuH1} and on the other hand, we treat the full range of subcritical nonlinearties which was not possible in the results from \cite{BouardHeins} and \cite{dissertationFH} for noise with low regularity. 
We further point out that beyond wellposedness results, the rescaling approach of Barbu, R\"ockner, and Zhang has been succesfully applied
to investigate blow-up and scattering \cite{BarbuNoBlowUp,HerrScattering} and control problems \cite{barbu2018optimal,zhang2019optimal}.    
To complement the existence result from \cite{ExistencePaper}, Brzezniak, Weis, and the author showed pathwise uniqueness on 3D compact manifolds, see \cite{brzezniak2018uniqueness} and also \cite{dissertationFH} for results in 2D. Approximation techniques have also been used by Keller and Lisei \cite{keller2015variational, lisei2016stochastic}. In contrast to \cite{ExistencePaper} and Theorem \ref{mainTheoremIntro}, however, their results are restricted to 1D bounded domains.
Let us point out that in  every article mentioned so far, the noise is considered to be white in time, but not in space. Very recently, other authors, most notably Gubinelli, Ugurcan, and Zachhuber \cite{gubinelli2018semilinear}, Debussche, Weber \cite{debussche2018schrodinger}, and Debussche, Martin \cite{debussche2019solution},  also constructed solutions to the NLS with purely spatial white noise combined with renormalization.
 Furthermore, we would like to mention existence results for jump noise by Brzezniak, Manna, and the author in \cite{Brzeźniak2019} and by de Bouard, Hausenblas, and Ondrejat in \cite{de2019nonlinear,de2019uniqueness}.

The paper is organized as follows. After  formulating the assumptions in Section 2, we concentrate on the wellposedness and the uniform estimates for the truncated equation in Section 3. 
In Section 4, a tightness criterion is derived which helps to pass to the limit in Section 5. In Section 6, we  finally, we present some examples which are covered by the general framework and thereby, we prove Theorem \ref{mainTheoremIntro}.

\subsection*{Notational remarks}

Before we start with the mathematical content of the paper, let us briefly recall some notations which will be frequently used.

\begin{itemize}
	\item We consider a finite time horizon $T>0$ and we assume that $\left(\Omega,\F,\Filtration, \Prob\right)$ is a \emph{filtered probability space with the usual conditions}.
	\item If $(A,\mathcal{A})$ is a measurable space and $X:\Omega\to A$ is a random variable, then the law of $X$ on $A$ is denoted by $\Prob^X.$ 
	\item If functions $a,b\ge 0$ satisfy the inequality $a\le C(A) b$ with a constant $C(A)>0$ depending on an expression $A$, we write $a \lesssim b$ and sometimes $a\lesssim_A\,b$ if the dependence on $A$ shall be highlighted. Given $a\lesssim b$ and $b\lesssim a,$ we write $a\eqsim b.$
	\item For two   Banach spaces $E,F$ over $\K\in\{\R,\C\},$ we denote by $\mathcal{L}(E,F)$ the space of linear bounded operators $B: E\to F$ and abbreviate $\mathcal{L}(E):=\mathcal{L}(E,E)$ as well as $E^*:=\LinearOperatorsTwo{E}{\K}.$  We write 
	\begin{align*}
	\duality{x}{x^*}:=x^*(x),\qquad x\in E,\quad x^*\in E^*.
	\end{align*}
\end{itemize}

%% file: Assumptions.tex
\section{The Setting}

	Let $\left(X,\Sigma,\mu\right)$ be a $\sigma$-finite measure space with metric $\rho$ satisfying the \emph{doubling property}, i.e. $\mu(B(x,r))<\infty$ for all $x\in X$ and $r>0$ and
	\begin{align}\label{doubling}
	\mu(B(x,2r))\lesssim \mu(B(x,r)).
	\end{align}
	 Let $M\subset X$ be an open subset and denote $H=L^2(M).$ The standard complex inner product on $H$ is denoted by 
	 \begin{align*}
	 \skpH{u}{v}=\int_M u \bar{v}\,\mu(\df x), \qquad u,v\in H. 
	 \end{align*}
	 Let $A$ be a $\C$-linear non-negative selfadjoint operator on ${H}$ with domain $\mathcal{D}(A)$.
	 	By $X_\theta$, $\theta\in\R$, we denote the scale of fractional domains of $A$, i.e. $X_\theta:=\D((\Id+A)^\theta)$ for $\theta\ge 0$ equipped with the norm
	 				\begin{align*}
	 				\norm{x}_\theta:=\norm{(\Id+A)^\theta x}_{H},\qquad x\in X_\theta,
	 				\end{align*}	
	 				and  $X_{-\theta}$, $\theta\ge 0$,  is the completion of $H$ with respect to the norm
	 				\begin{align*}
	 				\norm{x}_{-\theta}:=\norm{(\Id+A)^{-\theta}x}_{H},\qquad x\in H.
	 				\end{align*}
%
%
%
	 In the context of the NLS, it is necessary  that all our function spaces consist of $\C$-valued functions. However, in view of the stochastic integration theory and the compactness results we will present in Section \ref{CompactnessSection}, it is more convenient to interpret these spaces as real Hilbert or Banach spaces. Hence, we often treat $H$  as real a Hilbert space with the inner product $\Real (u,v)_H$ for $u,v\in H.$ Obviously, the real and the complex inner product induce the same norms and hence, both spaces are topologically  equivalent. 
	 The  Hilbert space $\EA:=X_{1/2}$ with
	 \begin{align*}
	 \skp{u}{v}_\EA:= \skpH{\left(\Id+A\right)^\frac{1}{2}u}{\left(\Id+A\right)^\frac{1}{2}v},\qquad u,v\in \EA,
	 \end{align*}
	 is called the \emph{energy space} \index{energy space} and $\norm{\cdot}_{E_A}$ the  \emph{energy norm} \index{energy norm}associated to $A.$ We further use the notation  $\EAdual:=X_{-\frac{1}{2}}$ which is justified since it is a classical result that $X_{1/2}$ and $X_{-1/2}$ are dual spaces.
	 We remark that $\left(E_A, H, E_A^*\right)$ is a Gelfand triple, i.e.
	 \begin{align*}
	 E_A\hookrightarrow H \cong H^* \hookrightarrow E_A^*,
	 \end{align*}
	 and recall that $A_{-1/2}$, i.e., the extension of $A$ to $\EAdual$, is a non-negative selfadjoint operator on $\EAdual$ with domain $\EA$ (cf., e.g.,  Proposition A.41 in \cite{dissertationFH}). For simplicity, we also denote  $A_{-\frac{1}{2}}$ by $A.$ Similarly to $H,$ the spaces $\EA$ and $\EAdual$ can be interpreted as real Hilbert spaces.

	\begin{Assumption}\label{spaceAssumptions}
		We assume the following: 
		\begin{enumerate}[(i)]
			\item  There is a strictly positive selfadjoint operator $S$ on ${H}$  commuting with $A$ which fulfills $\mathcal{D}(S^\rho)=\mathcal{D}(A)$ for some $\rho>0$.
			 Moreover, we assume that $S$  has \emph{generalized Gaussian $(p_0,p_0')$-bounds} for some $p_0\in [1,2),$ i.e. 
			\begin{align}\label{generalizedGaussianEstimate}
			\Big\Vert \mathbf{1}_{B(x,t^\frac{1}{m})}e^{-tS}\mathbf{1}_{B(y,t^\frac{1}{m})}\Big\Vert_{\mathcal{L}(L^{p_0},L^{p_0'})} \le C{\mu(B(x,t^\frac{1}{m}))}^{\frac{1}{p_0'}-\frac{1}{p_0}} \exp \left\{-c \left(\frac{\rho(x,y)^m}{t}\right)^{\frac{1}{m-1}}\right\}
			\end{align}
			for all $t>0$ and  $(x,y)\in M\times M$ with constants $c,C>0$ and $m\ge 2.$
			\item  There is a sequence $(M_n)_{n\in\N}\subset \mathcal{B}(M)$ which satisfies that $M_1\subset M_{2}\subset\dots$, $\bigcup_{n\in\N} M_n=M$, and that for all $(u_k)_{k\in\N}\subset \EA$, $u\in \EA$ it holds that 
			\begin{align*}
				\big(u_k\rightharpoonup u\quad\text{in $\EA$ for $k\to \infty$}\big)\quad \Rightarrow \quad \big(\forall n\in\N\colon\quad u_k\to u\quad\text{in $L^{\alpha+1}(M_n)$ for $k\to \infty$}\Big).
			\end{align*}
			\item Let $\alpha \in (1,p_0'-1)$, $\gamma\in [0,1/2)$ satisfy that  $X_\gamma$ is continuously embedded in $\LalphaPlusEins.$ 
		\end{enumerate}
	\end{Assumption}	 
	 
	 We point out that item (ii) in Assumption~\ref{spaceAssumptions} is a far reaching generalization of the assumption of compactness of the embedding $\EA\hookrightarrow\LalphaPlusEins$ which was present in \cite{ExistencePaper} and \cite{dissertationFH}. In fact, this assumption is weak enough to cover Laplacians on unbounded domains, the full space $\Rd$, and some non-compact manifolds as we will show in Section~\ref{ExamplesSection}. 
	 In the following, we abbreviate the real duality in $\EA$ with $\Real\duality{\cdot}{\cdot}:= \Real\duality{\cdot}{\cdot}_{\frac{1}{2},-\frac{1}{2}}.$  Note that the duality between $\LalphaPlusEins$ and $\LalphaPlusEinsDual$ given by 
	 \begin{align*}
	 \duality{u}{v}_{L^{\alpha+1},L^\frac{\alpha+1}{\alpha}}:=\int_M u\bar{v}\,\df \mu,\qquad u\in \LalphaPlusEins,\quad v\in \LalphaPlusEinsDual,
	 \end{align*}
	 extends $\duality{\cdot}{\cdot}$  in the sense that we have
	 \begin{align*}
	 \duality{u}{v}= \duality{u}{v}_{L^{\alpha+1},L^\frac{\alpha+1}{\alpha}},\qquad u\in \EA,\quad v\in \LalphaPlusEinsDual.
	 \end{align*}
	 We now proceed by specifying the class of nonlinearities which is admissible in our framework.
\begin{Assumption}\label{nonlinearAssumptions}
	Let $\alpha\in(1,p_0'-1)$ and $(M_n)_{n\in\N}\subset \mathcal{B}(M)$ be chosen as in Assumption $\ref{spaceAssumptions}.$ Then, we assume the following:
	\begin{itemize}
		\item[i)] Let $F: \LalphaPlusEins \to \LalphaPlusEinsDual$ be a function satisfying the following estimate
		\begin{align}\label{nonlinearityEstimate}
		\norm{F(u)}_\LalphaPlusEinsDualNorm \lesssim \norm{u}_\LalphaPlusEinsKurz^\alpha,\quad u\in \LalphaPlusEins.
		\end{align}
		We further assume  $F(0)=0$ and
		\begin{align}\label{nonlinearityComplex}
		\Real \duality{\im \mathbf{1}_{M_n} u}{ F(u)}=0, \quad u\in \LalphaPlusEins, \, n\in\N.
		\end{align}
		\item[ii)] The map $F: L^{\alpha+1}(M)\to L^\frac{\alpha+1}{\alpha}(M)$ is continuously real Fr\'{e}chet differentiable with
		\begin{align}\label{deriveNonlinearBound}
		\Vert F'[u]\Vert_{\mathcal{L}(\LalphaPlusEinsKurz,\LalphaPlusEinsDualNorm)} \lesssim \norm{u}_\LalphaPlusEinsKurz^{\alpha-1}, \quad u\in \LalphaPlusEins.
		\end{align}
		\item[iii)] The map $F$ has a real antiderivative $\hat{F},$ i.e. there exists a Fr\'{e}chet-differentiable map \\ $\hat{F}: \LalphaPlusEins\to \R$ with
		\begin{align}\label{antiderivative}
		\Fhat'[u]h=\Real \duality{F(u)}{h},\quad u,h\in \LalphaPlusEins.
		\end{align}
	\end{itemize}
\end{Assumption}	 


	Note that Assumptions~\ref{spaceAssumptions} and \ref{nonlinearAssumptions} lead to $F\in C(\EA, \EAdual)$ since we have the continuous embeddings $\EA\hookrightarrow\LalphaPlusEins$ and $(\LalphaPlusEins)^*=\LalphaPlusEinsDual\hookrightarrow \EAdual$ as well as $F\in C(\LalphaPlusEins,\LalphaPlusEinsDual)$. Moreover, we will often use that  Assumption $\ref{nonlinearAssumptions}$ ii) and the mean value theorem for Fr\'{e}chet differentiable maps imply
	\begin{align*}
	\norm{  F(x)-F(y)}_{\LalphaPlusEinsDualNorm}
	&\lesssim \left(\norm{x}_\LalphaPlusEinsKurz+\norm{y}_\LalphaPlusEinsKurz\right)^{\alpha-1} \norm{x-y}_\LalphaPlusEinsKurz,	\quad x,y\in \LalphaPlusEins.	
	\end{align*}	
 We will cover the following two standard types of nonlinearities.
 \begin{Definition}\index{nonlinearity!focusing}\index{nonlinearity!defocusing}
 	Let $F$ satisfy Assumption $\ref{nonlinearAssumptions}.$
 	Then, $F$ is called \emph{defocusing}, if $\Fhat(u)\ge 0$  and \emph{focusing}, if $\Fhat(u)\le 0$ for all $u\in \LalphaPlusEins.$
 \end{Definition}	
\begin{Assumption}\label{focusing}
	We assume either i) or i'):
	\begin{itemize}
		\item[i)] Let $F$ be defocusing and satisfy
		\begin{align}\label{boundantiderivative}
		\norm{u}_\LalphaPlusEinsKurz^{\alpha+1}\lesssim \Fhat(u), \quad u\in \LalphaPlusEins.
		\end{align}
		\item[i')] Let $F$ be focusing and satisfy
		\begin{align}\label{boundantiderivativeFocusing1}
		-\Fhat(u)\lesssim\norm{u}_\LalphaPlusEinsKurz^{\alpha+1}, \quad u\in \LalphaPlusEins.
		\end{align}
		and there is $\theta \in (0,\frac{2}{\alpha+1})$ with
		\begin{align}\label{interpolationFocusing}
		\left({H},\EA\right)_{\theta,1}\hookrightarrow \LalphaPlusEins.
		\end{align}
	\end{itemize}
\end{Assumption}

Here $\left(\cdot,\cdot\right)_{\theta,1}$ denotes the real interpolation space and 
we remark that by \cite{TriebelInterpolationTheory}, Lemma 1.10.1, the embedding \eqref{interpolationFocusing} is equivalent to
\begin{align}\label{boundantiderivativeFocusing2}
\norm{u}_\LalphaPlusEinsKurz^{\alpha+1}\lesssim \norm{u}_H^{\beta_1} \norm{u}_\EA^{\beta_2},\quad u\in \EA.
\end{align}
for some $\beta_1>0$ and $\beta_2 \in (0,2)$ with $\alpha+1= \beta_1+\beta_2.$
Let us continue with the definitions and assumptions for the stochastic part.

\begin{Assumption}\label{stochasticAssumptions}
	We assume the following:
	\begin{itemize}
		\item[i)] Let $(\Omega,\F,\Prob)$ be a complete probability space, $Y$ a separable real Hilbert space with ONB $(f_m)_{m\in\N}$ and $W$ a $Y$-valued cylindrical Wiener process adapted to a filtration $\Filtration$ satisfying the usual conditions.
		\item[ii)] Let $B: {H} \to \HS(Y,{H})$ be a linear
		operator and set $B_m u:=B(u)f_m$ for $u\in {H}$ and $m\in \N.$ Additionally, we assume that $B_m\in\mathcal{L}(H)$ is selfadjoint for every $m\in\N$ with
		\begin{align}\label{noiseBoundsH}
		\sumM \norm{B_m}_{\mathcal{L}({H})}^2<\infty
		\end{align}
		and assume $B_m\in \mathcal{L}(\EA)$ and $B_m\in \mathcal{L}(\LalphaPlusEins)$ for $m\in\N$ and $\alpha\in (1,p_0'-1)$ as in Assumption and Notation \ref{spaceAssumptions} with
		\begin{align}\label{noiseBoundsEnergy}
		\sumM\norm{B_m}_{\mathcal{L}(\EA)}^2<\infty,\hspace{1cm}
		\sumM \norm{B_m}_{\mathcal{L}(L^{\alpha+1})}^2<\infty.
		\end{align}
	\end{itemize}
	
\end{Assumption}	 

Under the previous assumptions, 
we investigate the following stochastic evolution equation with the Stratonovich noise
\begin{equation}\label{ProblemStratonovich2}
\left\{
\begin{aligned}
\df u(t)&= \left(-\im A u(t)-\im  F(u(t)\right) dt-\im B u(t) \circ \df W(t),\hspace{0,3 cm} t\in (0,T),\\
u(0)&=u_0,
\end{aligned}\right.
\end{equation}
where the stochastic differential is defined by	
\begin{align}
-\im B u(t) \circ\df W(t)=-\im B u(t) \df W(t)-\frac{1}{2}\sumM B_m^2 u(t) \df t.
\end{align}
Below, we will always use the It\^o form of this equation which reads 	
\begin{equation}\label{Problem}
\left\{
\begin{aligned}
\df u(t)&= \left(-\im A u(t)-\im  F(u(t))+ \mu \left(u(t)\right) \right) \df t-\im B u(t) \df W(t),\hspace{0,3 cm} t\in (0,T),\\
u(0)&=u_0,
\end{aligned}\right.
\end{equation}
where the linear operator $\mu$ is defined by
\begin{align*}
\mu(u) := -\frac{1}{2} \sumM B_m^2 u,\qquad u\in{H}.
\end{align*}
We next introduce the concept of a martingale solution to the problem \eqref{Problem}.

\begin{Definition}\label{SolutionDefStochastics}
	Let $T>0$, $u_0\in \EA.$ Then,
	a system $\left(\Omega',\F',\Prob',W',\Filtration',u\right)$ with
	\begin{itemize}
		\item a probability space $\big({\Omega}',{\F}',{\Prob}'\big);$
		\item  a filtration ${\Filtration}'=\big({\F}'_t\big)_{t\in [0,T]}$ with the usual conditions;
		\item a $Y$-valued cylindrical Wiener ${W}'$ process on $\tilde{\Omega}$ adapted to $\tilde{\Filtration};$
		\item a process $u: \Omega'\times [0,T]\to \EA$ which is adapted and continuous in $X_\theta$ for all $\theta\in (-\infty,1/2)$ and satisfies $u\in L^\infty(0,T;\EA)$ almost surely,
	\end{itemize}
	is called  \emph{global martingale solution} of \eqref{Problem} in $\EA$ if we have 
		\begin{align}\label{equationWithPseudoNonlin}
		u(t)= u_0+ \int_0^t \left[-\im A u(s)-\im  F(u(s))+\mu(u(s))\right] \df s-\im \int_0^t B(u(s))\df W'(s)
		\end{align} 
	almost surely for all $t\in [0,T]$.
\end{Definition}

We remark that martingale solutions are also often called \emph{stochastically weak solutions}. Note that our solution concept is also somehow weak in the analytical sense, since the equation is posed in $\EAdual$ and we do not ask for continuity of the solution in $\EA$.  

%% file: Truncation.tex
\section{Uniform estimates for a truncated equation}

We first construct the truncation operators we will use to approximate \eqref{Problem}. Later, we will formulate approximative equations and show that they have unique global solutions. Uniform estimates for these solutions are then deduced as a foundation of a limiting procedure.

Let us take a function $\rho\in C_c^\infty(0,\infty)$ with $\operatorname{supp} \rho \subset [\frac{1}{2},2]$ and $\sum_{m\in\Z} \rho(2^{-m} t)=1$ for all $t>0.$ For the existence of $\rho$ with these properties, we refer to \cite{bergh1976interpolation}, Lemma 6.1.7. 
Then, we fix $n\in\N_0$ and define 
\begin{align*}
s_n: (0,\infty)\to \R,\qquad s_n(\lambda):=\sum_{m=-\infty}^{n}\rho(2^{-m}\lambda).
\end{align*}
Let $k\in \Z$ and $\lambda\in [2^{k-1},2^k).$  From $\operatorname{supp} \rho \subset [\frac{1}{2},2],$ we infer
\begin{align*}
1=\sum_{m=-\infty}^\infty \rho(2^{-m}\lambda)=\rho(2^{-(k-1)}\lambda)+\rho(2^{-k}\lambda)=\sum_{m=-\infty}^k \rho(2^{-m}\lambda).
\end{align*}
In particular
\begin{align}\label{SnMultiplier}
s_n(\lambda)=\begin{cases}
1,\hspace{2cm}   &\lambda\in (0,2^n),\\
\rho(2^{-n}\lambda), &\lambda\in [2^{n},2^{n+1}),\\
0,&\lambda\ge 2^{n+1}.\\
\end{cases}
\end{align}
Moreover, we define $p_n\colon (0,\infty)\to \R$, $n\in\N_0$, by $p_n(\lambda)=\mathbf{1}_{(0,2^{n+1})}(\lambda)$ for $\lambda\in (0,\infty)$. In Figure \ref{FigureSnPn}, we display the functions $p_n,s_n:(0,\infty)\to \R$. Now, the operator sequences $(P_n)_{n\in\N_0}$ and $(S_n)_{n\in\N_0}$ are fixed by 
\begin{align}
	P_n:=p_n(S)\qquad \text{and}\qquad S_n:=s_n(S)
\end{align}
for $n\in \N_0$ using the Borel functional calculus of $S$.

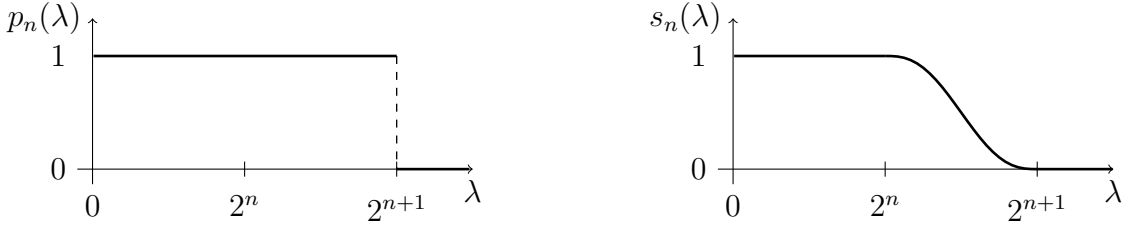
\begin{figure}[h]
	\begin{minipage}{0.49\textwidth}

		\begin{tikzpicture}
		
		
		\draw[->] (-0.2,0) -- (5,0) node[anchor=north] {$\lambda$};
		\draw[->] (0,-0.2) -- (0,2) node[anchor=east] {$p_n(\lambda)$};
		
		
		\draw	(-0.2,0) node[anchor=east] {$0$}
		(-0.2,1.5) node[anchor=east] {$1$}
		(0,-0.2) node[anchor=north] {$0$}
		(2,-0.2) node[anchor=north] {$2^n$}
		(4,-0.2) node[anchor=north] {$2^{n+1}$};
		
		\draw (2,-0.1)--(2,0.1);
		\draw (4,-0.1)--(4,0.1);
		
		
		\draw [line width=1pt](0.014,1.5)--(4,1.5);
		\draw [ dashed,line width=0.5pt](4,1.5)--(4,0.002);
		\draw [line width=1pt](3.991,0)--(4.95,0);
		\end{tikzpicture}
	\end{minipage}
	\begin{minipage}{0.49\textwidth}

		\begin{tikzpicture}
		
		\draw[->] (-0.2,0) -- (5,0) node[anchor=north] {$\lambda$};
		\draw[->] (0,-0.2) -- (0,2) node[anchor=east] {$s_n(\lambda)$};

		
		\draw	(-0.2,0) node[anchor=east] {$0$}
		(-0.2,1.5) node[anchor=east] {$1$}
		(0,-0.2) node[anchor=north] {$0$}
		(2,-0.2) node[anchor=north] {$2^n$}
		(4,-0.2) node[anchor=north] {$2^{n+1}$};
		\draw (2,-0.1)--(2,0.1);
		\draw (4,-0.1)--(4,0.1);
		
		
		\draw [line width=1pt](0.014,1.5)--(2,1.5);
		\draw [line width=1pt](3.991,0)--(4.99,0);
		\draw[domain=2:4, variable=\x,line width=1pt] plot (\x,{1.5*(-6*(0.5*\x-1)^5+15*(0.5*\x-1)^4-10*(0.5*\x-1)^3+1)});
		
		\end{tikzpicture}
	\end{minipage}
	\caption{Functions $p_n$ and $s_n$ with $P_n=p_n(S)$ and $S_n=s_n(S).$}
\end{figure} 

\begin{Lemma}\label{PnPropertiesLzwei}
	We fix $n\in\N_0.$
	\begin{enumerate}[a)]
		\item\label{PnOrthogonal} $P_n$ is an orthogonal projection in $H$. In particular, it holds that $\norm{P_n}_{\LinearOperators{H}}\le 1$.
		\item\label{RangePn}  $ H_n:= P_n(H)$ is a closed subspace of $H$ and we have $H_n\subset \EA$  and $\norm{P_n}_{\LinearOperators{\EA}}\le 1.$
		\item\label{PnExtended} $P_n$ can be extended to an operator $P_n: \EAdual\to \EAdual$ with  $P_n(\EAdual)=H_n$,
		\begin{gather}
					\norm{P_n}_{\LinearOperators{\EAdual}}\le 1,\qquad \norm{P_n}_{\LinearOperatorsTwo{H}{\EA}}\le 2^{(n+1){\nicefrac{\rho}{2}}},\qquad  \norm{P_n}_{\LinearOperatorsTwo{\EAdual}{H}}\le 2^{(n+1){\nicefrac{\rho}{2}}},
					\nonumber\\\duality{v}{P_n v}\in \R, \qquad \duality{v}{P_n w}=\skpH{P_n v}{w}, \qquad v\in \EAdual, \quad w\in H.
		\end{gather} 
		\item\label{truncationA} The restriction $A|_{H_n}$ of $A\in \LinearOperatorsTwo{\EA}{\EAdual}$ to $H_n$ is a bounded operator from $H_n$ to $H_n$ with $\norm{A|_{H_n}}_{\LinearOperators{H_n}}\le 2^{(n+1)\rho}$. 
	\end{enumerate}	
\end{Lemma}

\begin{proof}
	Item~\eqref{PnOrthogonal} is a direct consequence of $[p_n(\lambda)]^2 =p_n(\lambda)$ for all $\lambda\in (0,\infty)$.
	Suppose that $P_n(H)\ni y_m\to y\in H$ for $m\to\infty$. Since item~\eqref{PnOrthogonal} implies $P_n|_{H_n}=\operatorname{I}_{H_n}$ and $\norm{P_n}_{\LinearOperators{H}}\le 1$, we infer  $P_n y_m=y_m\to y$ and $P_n y_m\to P_n y$ for $m\to\infty$. This implies $y=P_n y\in P_n(H)$ and thus, $P_n(H)$ is a closed subspace of $H$.
	From $\mathcal{D}(S^\rho)=\mathcal{D}(A)$ and the injectivity of $S$ (cf.\ Assumption~\ref{spaceAssumptions}) as well as complex interpolation, we infer $\mathcal{D}(S^{\nicefrac{\rho}{2}})=\mathcal{D}(A^{\nicefrac{1}{2}})=\EA$ and
	\begin{align}\label{energyS}
		\norm{\left(\Id+A\right)^{\nicefrac{1}{2}} x}_{H}\eqsim \norm{S^{\nicefrac{\rho}{2}}x}_{H},\qquad x\in \EA.
	\end{align} 
		Next note that for all $a>0$ it holds that $(0,\infty)\ni \lambda \mapsto \lambda^a p_n(\lambda)$ is a bounded function with
		\begin{align*}
		\sup_{\lambda\in (0,\infty)}\vert\lambda^a p_n(\lambda)\vert\le 2^{(n+1)a}.
		\end{align*}
		Hence, we get $\norm{S^a P_n}_{\LinearOperators{H}}\le 2^{(n+1)a}$. This and \eqref{energyS} prove
	\begin{align*}
	\norm{P_n x}_\normEA=\norm{\left(\Id+A\right)^\frac{1}{2} P_n x}_{H}\eqsim \norm{S^{\nicefrac{\rho}{2}} P_n x}_{H}\le 2^{(n+1)\nicefrac{\rho}{2}} \norm{x}_H<\infty  ,\qquad x\in H.
	\end{align*}
	This proves $H_n=P_n(H)\subset \EA$.
	 Due to $P_n=p_n(S)$, $P_n$ commutes with  $\left(\Id+A\right)^\frac{1}{2}$ and $\left(\Id+A\right)^{-\frac{1}{2}}$ since $S$ and $A$ commute by Assumption \ref{spaceAssumptions}. We obtain
	\begin{align}\label{PnEstimateHtoEA}
	\norm{P_n x}_\normEA=\norm{\left(\Id+A\right)^\frac{1}{2} P_n x}_{H}\le \norm{\left(\Id+A\right)^\frac{1}{2} x}_{H}=\norm{x}_\normEA,\qquad x\in \EA,
	\end{align} 
	which completes the proof of item~\eqref{RangePn}.
	The contractivity of $P_n$ and \eqref{energyS} further imply
	\begin{align*}
	\norm{P_n x}_\normEAdual=\norm{\left(\Id+A\right)^{-\frac{1}{2}} P_n x}_{H}\le \norm{\left(\Id+A\right)^{-\frac{1}{2}}x}_{H}=\norm{x}_\normEAdual,\qquad x\in H,
	\end{align*}
and 
	\begin{align*}
		\norm{P_n x}_H&=\norm{\left(\Id+A\right)^\frac{1}{2} P_n \left(\Id+A\right)^{-\frac{1}{2}} x}_{H}
		\eqsim\norm{S^{\nicefrac{\rho}{2}} P_n \left(\Id+A\right)^{-\frac{1}{2}} x}_{H}
		\\&\le 2^{(n+1){\nicefrac{\rho}{2}}} \norm{\left(\Id+A\right)^{-\frac{1}{2}} x}_{H}
		=2^{(n+1){\nicefrac{\rho}{2}}} \norm{x}_{\EAdual}, \qquad x\in H.
	\end{align*}
	By \eqref{PnEstimateHtoEA} and the two previous estimates, we can extend $P_n$ to an operator $P_n: \EAdual\to \EAdual$ with 
	\begin{align*}
		\norm{P_n}_{\LinearOperators{\EAdual}}\le 1,\qquad \norm{P_n}_{\LinearOperatorsTwo{H}{\EA}}\le 2^{(n+1){\nicefrac{\rho}{2}}},\qquad \text{and}\qquad \norm{P_n}_{\LinearOperatorsTwo{\EAdual}{H}}\le 2^{(n+1){\nicefrac{\rho}{2}}}.
	\end{align*}
	To show that $P_n(\EAdual)=H_n$, we take $y\in P_n(\EAdual)$ and note that there exists $(x_k)_{k\in\N}\subset H$ with 
	$P_n x_k\to y$ in $H$ . Thus, we obtain $y\in P_n(H)=H_n$ by item~\eqref{PnOrthogonal}. This implies $P_n(\EAdual)=H_n$.  
	 For $w\in H$ and $v\in \EAdual$ with $H\ni v_k\to v$ as $k\to\infty,$ we conclude
	\begin{align*}
	\duality{v}{P_n v}=\lim_{k\to\infty}\skpH{v_k}{P_n v_k}\in\R
	\end{align*} 
	and 
	\begin{align*}
	\duality{v}{P_n w}=\lim_{k\to\infty} \skpH{v_k}{P_n w}=\lim_{k\to\infty} \skpH{P_n v_k}{ w}=\skpH{P_n v}{w}.
	\end{align*}
	This shows item~\eqref{PnExtended}.
		The fact that $A P_nx=P_n Ax$ for $x\in \EA$ shows $A|_{{P_n(\EA)}}\subset P_n(\EAdual)=H_n$ and the estimate 
		\begin{align*}
		\norm{A P_n x}_{H}\lesssim \norm{S^\rho P_n x}_H=\norm{(S^\rho P_n) P_n x}_H\le 2^{(n+1)\rho} \norm{P_n x}_H, \qquad x\in H,
		\end{align*}
		proves $A|_{H_n}$ maps $H_n$ to $H_n$ boundedly. 
\end{proof}

Below, the operator $A|_{H_n}$ will typically also be denoted by $A$. We continue with the properties of the operators $S_n$, $n\in\N_0$.

\begin{Lemma}\label{PaleyLittlewoodLemma} 
	\begin{enumerate}[a)]
		\item\label{SnLzwei} For each $n\in\N_0$ the operator $S_n\colon H\to H$ is bounded with $\norm{S_n}_{{\mathcal{L}(H)}}\le 1$, selfadjoint, and satisfies $S_n(H)\subset H_n$. 
		\item\label{SnEnergy} For each $n\in\N_0$ the operator $S_n$ maps $\EA$ into itself and satisfies $\norm{S_n}_{\mathcal{L}(\EA)}\le 1$.
		\item\label{SnLp} For each $n\in\N_0$ the operator $S_n$ can be extended to a bounded operator $S_n\colon \LalphaPlusEins\to \LalphaPlusEins$ such that $\sup_{m\in\N_0} \norm{S_m}_{\mathcal{L}(L^{\alpha+1})}<\infty$.
		\end{enumerate}
\end{Lemma}

\begin{proof}
	\emph{Step 1.}
Since $s_n$ is real-valued and bounded by $1,$ the operator $S_n$ is selfadjoint and bounded with $\norm{S_n}_\LinearOperators{H}\le 1.$
The fact that for all $\lambda\in (0,\infty)$ we have $ [p_n(\lambda)] [s_n(\lambda)]=s_n(\lambda)$ ensures $ P_n S_n x=S_nx$ for all $x\in H$ and therefore, we get $S_n(H)\subset P_n(H)=H_n$.
Next observe that item~\eqref{SnEnergy} is a direct consequence of the assumption that 
 $S$ and $A$ commute.
	
	\emph{Step 2.} Next, we show item~\eqref{SnLp} based on a spectral multiplier theorem by Kunstmann and Uhl, \cite{kunstmannUhl}, for operators with generalized Gaussian bounds.
	In view of Theorem 5.3 in \cite{kunstmannUhl}, Lemma 2.19 and Fact 2.20 in \cite{Uhl}, it is sufficient to show that $s_n$ satisfies the Mihlin condition
	\begin{align}\label{MihlinFn}
	\sup_{\lambda>0} \vert \lambda^k s_n^{(k)}(\lambda)\vert\le C_k,\qquad k=0,\dots, \gamma,
	\end{align}
	for some $\gamma\in \N$ uniformly in $n\in\N_0.$  This  can be verified by the calculation
	\begin{align*}
	\sup_{\lambda>0} \vert \lambda^k s_n^{(k)}(\lambda)\vert=&\sup_{\lambda\in [2^{n},2^{n+1})} \vert \lambda^k s_n^{(k)}(\lambda)\vert=\sup_{\lambda\in [2^{n},2^{n+1})} \left\vert \lambda^k \frac{\df^k }{\df \lambda^k}\rho(2^{-n}\lambda)\right\vert\le 2^k \norm{\rho^{(k)}}_\infty
	\end{align*}
	for all $k\in \N_0.$ 
\end{proof}	

After having established some nice properties of each $P_n$ and $S_n$ for $n\in\N_0$, we continue with the limiting behaviour of these operators as $n\to \infty.$

\begin{Lemma}\label{convergenceProperty}
	Let $\theta\in (0,\infty)$ and $x\in X_\theta$. Then it holds that 
	\begin{align*}
		\norm{(I+A)^\theta (P_nx-x)}_H\to 0\qquad \text{and}\qquad \norm{(I+A)^\theta (S_n x-x)}_H\to 0.
	\end{align*}
\end{Lemma}

\begin{proof}
	Observe that for all $\lambda\in (0,\infty)$ it holds that $p_n(\lambda)\to 1$ and $s_n(\lambda)\to 1$ as well as $\vert p_n(\lambda)\vert \le 1$ and $\vert s_n(\lambda)\vert \le 1$.
	By the selfadjoint functional calculus of $S$, we thus get 
	\begin{align*}
		\norm{ (P_ny-y)}_H\to 0\qquad \text{and}\qquad \norm{ (S_n y-y)}_H\to 0
	\end{align*}
	for $y\in H$. The fact $(I+A)^\theta$ commutes with both $P_n$ and $S_n$ proves the assertion.
\end{proof}

	Using the operators $P_n$ and $S_n,$ $n\in\N_0,$ we approximate our original problem $\eqref{ProblemStratonovich}$ by the stochastic evolution equation in $H_n$ given by	
	\begin{equation*}
	\left\{
	\begin{aligned}
	\df u_n(t)&= \left(-\im A u_n(t)-\im P_n F\left( u_n(t)\right) \right) \df t-\im  S_n B(S_n u_n(t)) \circ \df W(t),
	\\
	u_n(0)&=S_n u_0.
	\end{aligned}\right.
	\end{equation*}
	We emphasize that the truncation of the initial value by $S_n$ could also be replaced by $P_n$, which has been used in \cite{ExistencePaper}. We choose $S_n$ here since this makes the energy estimates  slightly more natural in the defocusing case, see equation \eqref{energySn} below.  
	With the Stratonovich correction term
	\begin{align*}
	\mu_n := -\frac{1}{2} \sumM \left(S_n B_m S_n\right)^2,
	\end{align*}
	the approximated problem can also be written in the It\^o form 	
	\begin{equation}\label{galerkinEquation}
	\left\{
	\begin{aligned}
	\df u_n(t)&= \left(-\im A u_n(t)-i P_n F\left(  u_n(t)\right)+ \mu_n \left(u_n(t)\right) \right) \df t-\im  S_n B (S_n u_n(t)) \df W(t),
	\\
	u_n(0)&=S_n u_0.
	\end{aligned}\right.
	\end{equation}
	We remark that by the mapping properties of the operators $P_n$ and $S_n$ (cf.~Lemma~\ref{PnPropertiesLzwei} and Lemma~\ref{PaleyLittlewoodLemma}), \eqref{galerkinEquation} can be viewed as a stochastic differential equation in the Hilbert space $H_n=P_n(H)$ equipped with the scalar product inherited from $H$.  
	First, we are concerned with a local wellposedness result for $\eqref{galerkinEquation}.$ Note that this is not as simple as in \cite{ExistencePaper}, where the compactness of the embedding $\EA\hookrightarrow \LalphaPlusEins$ and the finite volume of $M$ ensured that $H_n$ was finite dimensional. The regularizing properties of $P_n$ from Lemma~\ref{PnPropertiesLzwei}, however, compensate this and ensure the local Lipschitz property of the truncated nonlinearity in the $H$-norm.

\begin{Prop}\label{localSolutionGalerkin}
	Fix $n\in\N_0.$ Then there is a unique maximal solution $\left(u_n,\left(\tau_{n,k}\right)_{k\in\N},\tau_n\right)$ of $\eqref{galerkinEquation}$ with continuous paths in $H_n,$ i.e. there is an increasing sequence $\left(\tau_{n,k}\right)_{k\in\N}$ of stopping times with $\tau_n=\sup_{k\in\N} \tau_{n,k}$
	and
	\begin{align}\label{integralFormulationGalerkin}
	u_n(t)= S_n u_0&+ \int_0^t \big[-\im A u_n(s)-\im P_n F\left( u_n(s)\right)+\mu_n(u_n(s))\big] \df s- \im \int_0^t S_n B (S_n u_n(s)) \df W(s)
	\end{align}
	almost surely on $\{t\le \tau_{n,k}\}$ for all $k\in \N.$
	Moreover, we have the blow-up criterion
	\begin{align}\label{BlowUpTime}
	\Prob( \tau_{n,k}<T\quad \forall k\in\N,\quad \sup_{t\in [0,\tau_n)}\norm{u_n(t)}_{H}<\infty)=0.
	\end{align}
\end{Prop}

\begin{proof}
	The assertion follows if we can show that the functions 
	$f_n: H_n \rightarrow H_n$ and $\sigma_n: H_n \rightarrow \HS(Y,H_n)$ defined by
	\begin{align*}
	f_n(x):=&-\im A x-i P_n F(x)+ \mu_n \left(x\right),\qquad 
	\sigma_n(x):=-\im  S_n B(S_n x),\qquad x\in H_n,
	\end{align*}
	are Lipschitz on balls in $H_n.$ Given $R>0$ and $x,y\in H_n$ with $\norm{x}_{H}\le R$ and $\norm{y}_{H}\le R$, we estimate
	\begin{align*}
	\norm{\mu_n(x)-\mu_n(y)}_H &\le \bigNorm{\sumM \left(S_n B_m S_n\right)^2 (x-y)}_{H}\\
	&\le \sumM \Vert B_m\Vert_{\LinearOperators{H}}^2 \Vert S_n\Vert_{\LinearOperators{H}}^4 \norm{x-y}_{H}
	\lesssim \norm{x-y}_H
	\end{align*}
	where we used \eqref{noiseBoundsEnergy}.
	The regularizing properties of $P_n$ from Lemma~\ref{PnPropertiesLzwei}, the local Lipschitz estimate of $F$ and the fact that $x=P_nx$ and $y=P_ny$ yield
	\begin{align*}
	&\norm{P_n F(x)-P_n F(y)}_{H}
	\lesssim_n \norm{ F(x)- F(y)}_{\EAdual}\lesssim			\norm{F(x)-F(y)}_{\LalphaPlusEinsDualNorm}\\
	&\lesssim 	  \left(\norm{x}_{L^{\alpha+1}}+\norm{y}_{L^{\alpha+1}}\right)^{\alpha-1} \norm{x-y}_{L^{\alpha+1}}	
	\lesssim  	  \left(\norm{x}_{\EA}+\norm{y}_{\EA}\right)^{\alpha-1} \norm{x-y}_{\EA}
	\\&=\left(\norm{P_n x}_{\EA}+\norm{P_n y}_{\EA}\right)^{\alpha-1} \norm{P_n(x-y)}_{\EA}
	\lesssim_n  	  \left(\norm{x}_{H}+\norm{y}_{H}\right)^{\alpha-1} \norm{x-y}_{H}
	\\&\le 	  2^{\alpha-1} R^{\alpha-1} \norm{x-y}_{H}.
	\end{align*}		
	Using also item~\eqref{truncationA} in Lemma~\ref{PnPropertiesLzwei}, we deduce
	\begin{align*}
	\norm{f_n(x)-f_n(y)}_{H}\lesssim_{n,R} \norm{x-y}_{H}.			
	\end{align*}
	 From \eqref{noiseBoundsEnergy}, we infer
	\begin{align*}
	\norm{\sigma_n(x)-\sigma_n(y)}_{\HS(Y,H_n)}^2&=\sumM \norm{S_nB_m S_n (x-y)}_{H}^2\\
	&\le \left(\sumM \norm{ B_m}_{\LinearOperators{H}}^2 \right) \Vert S_n\Vert_{\LinearOperators{H}}^4 \norm{ x-y}_{H}^2\lesssim \norm{x-y}_{H}^2.
	\end{align*}
\end{proof}

The following shows that the local solutions of the truncated equations in fact exist for all times due to mass conservation. We skip the proof since it is identical with the analogous result in \cite{ExistencePaper}, Proposition 5.4. 

\begin{Prop}\label{MassEstimateGalerkinSolution}
	For each $n\in\N_0,$ there is a unique global solution $u_n$ of $\eqref{galerkinEquation}$ with continuous paths in $H_n$ and we have the estimate
	\begin{align}\label{LzweiEstimate}
	\norm{u_n(t)}_{H}=\norm{S_n u_0}_{H}\le \norm{u_0}_{H}
	\end{align}
	almost surely for all $t\in [0,T].$
\end{Prop}

	\begin{Definition}
		We define the energy $\energy(u)$ of $u\in \EA$ by
		\begin{align*}
		\energy(u):= \frac{1}{2} \Vert A^{\frac{1}{2}} u \Vert_{H}^2+\Fhat(u),\qquad u\in \EA.
		\end{align*}	
	\end{Definition}
	
Besides the $L^2$-conservation, the main structural feature of the deterministic NLS is the energy conservation. The stochastic noise, however, destroys this property. Fortunately, though, it is possible to generalize the energy conservation by a Gronwall type argument.
In order to use energy estimates to deduce uniform estimates in $\EA$, we distinguish between defocusing and focusing nonlinearities.
In fact, we proved the same  estimates already in \cite{ExistencePaper} and (with nonlinear noise and a slightly different truncated equation) in \cite{dissertationFH} without using the additional assumption on the compactness of the embedding $\EA\hookrightarrow \LalphaPlusEins$ which was only needed in different parts of the existence proofs there. However, we repeat the argument at least in the case of a defocusing nonlinearity in order to convince the reader that only the properties of $P_n$ and $S_n$ as stated in   Lemma~\ref{PnPropertiesLzwei} and in Lemma~\ref{PaleyLittlewoodLemma} are used.
	
	\begin{Prop}\label{EstimatesGalerkinSolutionDefocusing}
		Under Assumption \ref{focusing} i), the following assertions hold:
		\begin{enumerate}
			\item[a)] For all $q\in [1,\infty)$ there is  
			$C>0$ with
			\begin{align}\label{defocusingEnergyEstimate}
			\sup_{n\in\N_0}\E \Big[\sup_{t\in[0,T]} \left[\norm{u_n(t)}_{H}^2+\energy(u_n(t))\right]^q\Big]\le C.
			\end{align}
			In particular, for all $r\in[1,\infty)$ there is 
			$C_1>0$ with
			\begin{align}\label{defocusingHoneEstimate}
			\sup_{n\in\N_0}\E \Big[\sup_{t\in[0,T]} \norm{u_n(t)}_\EA^{r}\Big]\le C_1.
			\end{align}
			\item[b)] For all $r\in[1,\infty)$ there is 
			$C_2>0$ with
							\begin{align}\label{defocusingNonlinearityEstimate}
							\sup_{n\in\N_0}\E \Big[\sup_{t\in[0,T]} \norm{F(u_n(t))}_{\LalphaPlusEinsDualNorm}^{r}\Big]\le C_2.
							\end{align}
		\end{enumerate}
	\end{Prop}

	\begin{proof}
It\^o's formula and Proposition $\ref{MassEstimateGalerkinSolution}$ lead to the identity
		\begin{align}\label{ItoEnergyStart}
		\norm{u_n(t)}_{H}^2+\energy\left(u_n(t)\right)=&\norm{S_n u_0}_{H}^2+\energy\left(S_n u_0\right)
		\nonumber\\
		&+ \int_0^t\Real \duality{A u_n(s)+F(u_n(s))}{ -\im A u_n(s)-\im P_n F(u_n(s))}\df s\nonumber \\
		&+\int_0^t
		\Real \duality{A u_n(s)+F(u_n(s))}{ \mu_n(u_n(s))} \df s\nonumber \\
		&+\int_0^t \Real \duality{A u_n(s)+F(u_n(s))}{ -\im S_n B \left(S_n u_n(s)\right)\df W(s)}\nonumber \\
		&+\frac{1}{2}\sumM \int_0^t  \Vert \sqrtA S_n B_m S_n u_n(s)\Vert_{H}^2\df s\nonumber\\
		&+\frac{1}{2} \int_0^t \sumM \Real \duality{F'[u_n(s)] \left(S_n B_m S_n u_n(s)\right)}{ S_n B_m S_n u_n(s)} \df s
		\end{align}
		almost surely for all $t\in [0,T].$
		Lemma~\ref{PnPropertiesLzwei} c) \& d) now proves for all $v\in H_n$ that
		\begin{align*}
		\Real \duality{F(v)}{ -\im P_n F(v)}=\Real \left[\im \duality{F(v)}{  P_n F(v)}\right]=0;
		\end{align*}
		\begin{align*}
		\Real \left[\duality{A v}{ -\im P_n F(v)}+\duality{F(v)}{ -\im A v}\right]
		&=\Real \left[-\duality{A v}{ \im F(v)}+\overline{\duality{  A v}{\im F(v)}}\right]=0;
		\end{align*}
		\begin{align*}
		\Real \skpH{A v}{ -\im A v}=\Real \left[\im \norm{A v}_{H}^2\right]=0.
		\end{align*}
		 These identities simplify $\eqref{ItoEnergyStart}$ and we get				
		\begin{align}\label{ItoEnergyWithoutExponent}
		\norm{u_n(t)}_{H}^2+\energy\left(u_n(t)\right)=&\norm{S_n u_0}_{H}^2+\energy\left(S_n u_0\right)
		+\int_0^t \Real \duality{A u_n(s)+F(u_n(s))}{ \mu_n(u_n(s))} \df s\nonumber \\
		&+\int_0^t \Real \duality{A u_n(s)+F(u_n(s))}{ -\im S_n B \left(S_n u_n(s)\right)\df W(s)}\nonumber \\
		&+\frac{1}{2}\sumM \int_0^t  \Vert \sqrtA S_n B_m S_n u_n(s)\Vert_{H}^2\df s\nonumber\\
		&+\frac{1}{2}\int_0^t \sumM \Real \duality{F'[u_n(s)] \left(S_n B_m S_n u_n(s)\right)}{ S_n B_m S_n u_n(s)} \df s
		\end{align}	
		almost surely for all $t\in [0,T].$ 												
		Next, we fix $\delta>0,$ $q>1$ and apply \dela{It\^{o}'s}{the It\^o} formula to the process on the LHS of $\eqref{ItoEnergyWithoutExponent}$ and the function $\varPhi: (-\frac{\delta}{2},\infty)\to \R$ defined by $\varPhi(x):=\left(x+\delta\right)^q.$ 
		With the short notation
		\begin{align*}
		Y(s):=\delta+\norm{u_n(s)}_{H}^2+\energy\left(u_n(s)\right),\qquad s\in [0,T],
		\end{align*}
		we obtain
		\begin{align}\label{ItoEnergy}
		Y(t)^{q}=&\left[\delta+\norm{S_n u_0}_{H}^2+\energy\left(S_n u_0\right)\right]^q
		+q\int_0^t Y(s)^{q-1} \Real \duality{A u_n(s)+F(u_n(s))}{ \mu_n(u_n(s))} \df s\nonumber \\
		&+q\int_0^t Y(s)^{q-1} \Real \duality{A u_n(s)+F(u_n(s))}{ -\im S_n B \left(S_n u_n(s)\right)\df W(s)}\nonumber \\
		&+\frac{q}{2}\sumM \int_0^t  Y(s)^{q-1} \Vert \sqrtA S_n B_m S_n u_n(s)\Vert_{H}^2\df s\nonumber\\
		&+\frac{q}{2} \sumM \int_0^t  Y(s)^{q-1} \Real \duality{F'[u_n(s)] \left(S_n B_m S_n u_n(s)\right)}{ S_n B_m S_n u_n(s)} \df s\nonumber\\
		&+\frac{q}{2}(q-1)\sumM \int_0^t Y(s)^{q-2} \left[\Real \duality{A u_n(s)+F(u_n(s))}{-\im S_n B_m S_n u_n(s)}\right]^2 \df s
		\end{align}					
		almost surely for all $t\in [0,T].$
		In order to treat the stochastic integral, we  use Lemma  $\ref{PaleyLittlewoodLemma}$ and Proposition $\ref{MassEstimateGalerkinSolution}$ to estimate for fixed $s\in [0,T]$
		\begin{align}\label{energyEstimateEins}
		\vert\skpH{A u_n(s)}{-\im   S_n B_m S_n u_n(s)}\vert
		&\le \Vert \sqrtA u_n(s)\Vert_{H} \Vert \sqrtA S_n B_m S_n u_n(s) \Vert_{H}\nonumber\\
		&\le \Vert \sqrtA u_n(s)\Vert_{H} \Vert   S_n B_m S_n u_n(s)\Vert_\EA\nonumber\\
		&\le \Vert \sqrtA u_n(s) \Vert_{H} \norm{S_n}_{\mathcal{L}(\EA)}^2\Vert B_m\Vert_{{\mathcal{L}(\EA)}} \Vert  u_n(s)\Vert_\EA\nonumber\\
		&\le \left( \norm{u_n(s)}_{H}^2+\Vert \sqrtA u_n(s) \Vert_{H}^2\right) \Vert B_m\Vert_{{\mathcal{L}(\EA)}}\nonumber\\
		&\lesssim Y(s) \Vert B_m\Vert_{{\mathcal{L}(\EA)}}
		\end{align}
		and $\eqref{nonlinearityEstimate},$ $\eqref{boundantiderivative}$ and Lemma $\ref{PaleyLittlewoodLemma}$ to estimate
		\begin{align}\label{energyEstimatezwei}
		\vert \duality{F(u_n(s))}{ -\im S_n B_m S_n u_n(s) }\vert
		&\le \Vert F(u_n(s))\Vert_{\LalphaPlusEinsDualNorm} \Vert  S_n B_m S_n u_n(s) \Vert_\LalphaPlusEinsKurz\nonumber\\
		&\le   \Vert u_n(s)\Vert_\LalphaPlusEinsKurz^{\alpha+1}  \Vert S_n\Vert_{\mathcal{L}(L^{\alpha+1})}^2 \Vert B_m\Vert_{\mathcal{L}(L^{\alpha+1})}\nonumber\\
		&\lesssim \Fhat(u_n(s))   \Vert B_m\Vert_{\mathcal{L}(L^{\alpha+1})}
		\nonumber\\
		&\lesssim Y(s)   \Vert B_m\Vert_{\mathcal{L}(L^{\alpha+1})}.
		\end{align}
		The Burkholder-Gundy-Davis inequality, the estimates $\eqref{energyEstimateEins}$ and $\eqref{energyEstimatezwei},$ Assumption $\ref{stochasticAssumptions}$ and Lemma 5.6 in \cite{ExistencePaper}  prove for any $\varepsilon>0$
		\begin{align}\label{burkholderEstimate1}
		\E \Big[ &\sup_{s\in[0,t]} \left\vert \int_0^s Y(r)^{q-1}\Real \duality{A u_n(r)+F(u_n(r))}{-\im S_n B  \left(S_n u_n(r)\right) \df W(r)}\right\vert\Big]\nonumber\\
		&\lesssim\E \Big[  \left(\int_0^t \sumM \left\vert Y(r)^{q-1} \Real \duality{A u_n(r)+F(u_n(r))}{-\im S_n B_m S_n u_n(r)}\right\vert^2 \df r\right)^{\frac{1}{2}}\Big]\nonumber\\
		&\lesssim \E \Bigg[  \Bigg(\int_0^t Y(r)^{2q}  \df r\Bigg)^{\frac{1}{2}}\Bigg]	
		\le \varepsilon \E\Big[ \sup_{s\in[0,{t}]} Y(s)^q\Big]+ \frac{1}{4 \varepsilon} \int_0^t \E \Big[\sup_{r\in[0,s]} Y(r)^q\Big] \df s.		
		\end{align}
		To estimate the integrands of the deterministic integrals,  we fix $s\in [0,T]$ and get
		\begin{align}\label{energyEstimateDrei}
		\Real \skpH{A u_n(s)}{ \left(S_n B_m S_n \right)^2 u_n(s)}
		&\le \Vert \sqrtA u_n(s)\Vert_{H} \Vert \sqrtA \left(S_n B_m S_n \right)^2 u_n(s)\Vert_{H}\nonumber\\
		&\le \Vert \sqrtA u_n(s)\Vert_{H} \Vert  \left(S_n B_m S_n\right)^2 u_n(s)\Vert_\EA\nonumber\\
		&\le \Vert \sqrtA u_n(s) \Vert_{H} \Vert S_n\Vert_{{\mathcal{L}(\EA)}}^4 \Vert B_m\Vert_{{\mathcal{L}(\EA)}}^2 \Vert  u_n(s)\Vert_\EA\nonumber\\
		&\le \left(\Vert u_n(s) \Vert_{H}^2+\Vert \sqrtA u_n(s) \Vert_{H}^2\right) \Vert B_m\Vert_{{\mathcal{L}(\EA)}}^2\nonumber\\
		&\lesssim Y(s) \Vert B_m\Vert_{{\mathcal{L}(\EA)}}^2; 
		\end{align}
		\begin{align}\label{energyEstimateVier}
		\Real \duality{F(u_n(s))}{ \left(S_n B_m S_n \right)^2 u_n(s)}
		&\le \Vert F(u_n(s))\Vert_{\LalphaPlusEinsDualNorm} \Vert \left(S_n B_m S_n\right)^2 u_n(s) \Vert_\LalphaPlusEinsKurz\nonumber\\
		&\lesssim  \Vert u_n(s)\Vert_\LalphaPlusEinsKurz^{\alpha+1} \Vert S_n\Vert_{\mathcal{L}(L^{\alpha+1})}^4 \Vert B_m\Vert_{\mathcal{L}(L^{\alpha+1})}^2\nonumber\\
		&\lesssim \Fhat(u_n(s))   \Vert B_m\Vert_{\mathcal{L}(L^{\alpha+1})}^2
		\lesssim Y(s) \Vert B_m\Vert_{\mathcal{L}(L^{\alpha+1})}^2;
		\end{align}				
		\begin{align}\label{energyEstimateFunf}
		\Vert \sqrtA S_n B_m S_n u_n(s)\Vert_{H}^2
		&\le\Vert  S_n B_m S_n u_n(s)\Vert_\EA^2
		\le \Vert S_n\Vert_{{\mathcal{L}(\EA)}}^4 \Vert B_m\Vert_{{\mathcal{L}(\EA)}}^2 \Vert u_n(s)\Vert_\EA^2\nonumber\\
		&\le \Vert B_m\Vert_{{\mathcal{L}(\EA)}}^2 \left(\Vert  u_n(s) \Vert_{H}^2+\Vert \sqrtA u_n(s)\Vert_{H}^2\right)\nonumber\\
		&\lesssim \Vert B_m\Vert_{{\mathcal{L}(\EA)}}^2 Y(s)
		\end{align}
		for $m\in\N$ and $s\in[0,T].$
		Moreover, note that
		\begin{align}\label{energyEstimateSechs}
		\Real \duality{F'[u_n(s)] \left(S_n B_m S_n u_n(s)\right)}{ S_n B_m S_n u_n(s)}
		&\lesssim \norm{F'[u_n(s)]}_{\LinearOperatorsTwo{\LalphaPlusEinsKurz}{\LalphaPlusEinsDualNorm}}  \Vert S_n B_m S_n  u_n(s)\Vert_\LalphaPlusEinsKurz^2\nonumber\\
		&\lesssim  \Vert u_n(s)\Vert_\LalphaPlusEinsKurz^{\alpha+1}  \Vert S_n\Vert_{\mathcal{L}(L^{\alpha+1})}^4 \Vert B_m\Vert_{\mathcal{L}(L^{\alpha+1})}^2\nonumber\\
		&\lesssim \Fhat(u_n(s))   \Vert B_m\Vert_{\mathcal{L}(L^{\alpha+1})}^2
		\lesssim Y(s)   \Vert B_m\Vert_{\mathcal{L}(L^{\alpha+1})}^2.
		\end{align}
		Substituting the inequalities $\eqref{burkholderEstimate1}$ to $\eqref{energyEstimateSechs},$   into the identity $\eqref{ItoEnergy},$ we get for each $t\in [0,T]$
		\begin{align}\label{GronwallInequality}
		\E \big[\sup_{s\in[0,t]}Y(s)^q\big]\lesssim_q& \left[\delta+\norm{S_n u_0}_{H}^2+\energy(S_n u_0)\right]^q		
		+\E \int_0^t  \sumM \norm{B_m}_{{\mathcal{L}(\EA)}}^2 Y(s)^q\df s\nonumber\\
		&+\E \int_0^t \sumM\norm{B_m}_{{\mathcal{L}(L^{\alpha+1})}}^2Y(s)^q \df s \nonumber\\
		&+\varepsilon \E  \Big[ \sup_{r\in[0,t]}Y(s)^q\Big] +\frac{1}{4 \varepsilon} \int_0^t \E\Big[\sup_{s\in[0,r]}  Y(s)^q\Big]  \df r\nonumber\\
		&+\E \sumM \int_0^t  \norm{B_m}_{{\mathcal{L}(\EA)}}^2 Y(s)^q \df s + \E \int_0^t\sumM\norm{B_m}_{{\mathcal{L}(L^{\alpha+1})}}^2Y(s)^q \df s\nonumber\\
		&+\E \int_{0}^t Y(s)^{q} \sumM \max\{\norm{B_m}_{{\mathcal{L}(\EA)}}^2,\norm{B_m}_{\mathcal{L}(L^{\alpha+1})}^2\} \df s\nonumber\\
		\lesssim& \left[\delta+\norm{u_0}_{H}^2+\energy(S_n u_0)\right]^q		
		+\E \int_0^t   Y(s)^q\df s\nonumber\\
		&+\varepsilon \E  \Big[ \sup_{r\in[0,t]}Y(s)^q\Big] +\frac{1}{4 \varepsilon} \int_0^t \E\Big[\sup_{s\in[0,r]}  Y(s)^q\Big]  \df r\nonumber\\
		\lesssim_T& \left[\delta+\norm{u_0}_{H}^2+\energy(S_n u_0)\right]^q
		+\varepsilon \E  \Big[ \sup_{r\in[0,t]}Y(s)^q\Big] + \int_0^t \E \Big[\sup_{s\in[0,r]} Y(s)^q\Big]  \df r.
		\end{align}
		Choosing $\varepsilon>0$ small enough in inequality $\eqref{GronwallInequality}$, the Gronwall lemma yields
		\begin{align*}
		\E \big[\sup_{s\in[0,t]}Y(s)^q\big]\le C\left[\delta+\norm{u_0}_{H}^2+\energy(S_n u_0)\right]^q e^{C t},\qquad t\in [0,T],
		\end{align*}
		with a constant $C>0$, which is uniform in $n\in\N_0.$ Since we are in the defocusing case, we obtain
		\begin{align}\label{energySn}
		\energy(S_n u_0)\lesssim \norm{\sqrtA S_n u_0}_{H}^2+\norm{S_n u_0}_\LalphaPlusEinsKurz^{\alpha+1}\lesssim \norm{\sqrtA u_0}_{H}^2+\norm{ u_0}_\LalphaPlusEinsKurz^{\alpha+1}\lesssim \energy(u_0),
		\end{align}
		and thus, we have proved the estimate \eqref{defocusingEnergyEstimate}. Since the assumption that $F$ is defocusing, i.e. $\Fhat(w)\ge 0$ for all $w\in\EA$, implies
		\begin{align*}
			\norm{w}_\EA^2\lesssim \norm{w}_{H}^2+\energy(w),\qquad w\in \EA,
		\end{align*}
		it follows that we also have \eqref{defocusingHoneEstimate}. To prove b), we employ Assumption~\ref{nonlinearAssumptions} and \eqref{defocusingHoneEstimate} to deduce
		\begin{align*}
		\sup_{n\in\N_0}\E \Big[\sup_{t\in[0,T]} \norm{F(u_n(t))}_{\LalphaPlusEinsDualNorm}^{r}\Big]
		\lesssim \sup_{n\in\N_0}\E \Big[\sup_{t\in[0,T]}  \norm{u_n(t)}_\LalphaPlusEinsKurz^{\alpha r}\Big]
		\lesssim \sup_{n\in\N_0}\E \Big[\sup_{t\in[0,T]}  \norm{u_n(t)}_\EA^{\alpha r}\Big]
		\lesssim 1.
		\end{align*}
	\end{proof}

As announced above, we skip the proof of the following result in the focusing case since it is similar to Proposition 5.8 in \cite{ExistencePaper}.
	
		\begin{Prop}\label{EstimatesGalerkinSolutionFocusing}
			Under Assumption $\ref{focusing}$ i'), the following assertions hold:
			\begin{enumerate}
				\item[a)]
				For all $r\in[1,\infty)$ there is 
$C_1>0$ with
				\begin{align*}
				\sup_{n\in\N_0}\E \Big[\sup_{t\in[0,T]} \norm{u_n(t)}_\EA^{r}\Big]\le C_1.
				\end{align*}
			\item[b)] For all $r\in[1,\infty)$ there is 
$C_2>0$ with
			\begin{align*}
			\sup_{n\in\N_0}\E \Big[\sup_{t\in[0,T]} \norm{F(u_n(t))}_{\LalphaPlusEinsDualNorm}^{r}\Big]\le C_2.
			\end{align*}
			\end{enumerate}
		\end{Prop}

Since the Arzela-Ascoli theorem is a powerful tool to prove compactness, equicontinuity typically plays an important role in compactness arguments. We present a variant for stochastic processes which is formulated via the notion of stochastic convergence.
	
\begin{Definition}[Aldous condition]\label{DefinitionAldous}
	Let $(X_n)_{n\in\N_0}$ be a sequence of adapted stochastic processes in a Banach space $E.$ Assume that for every $\varepsilon>0$ and $\eta>0$ there is $\delta>0$ such that for every sequence $(\tau_n)_{n\in\N_0}$ of $[0,T]$-valued stopping times one has
	\begin{align*}
	\sup_{n\in\N_0} \sup_{0<\theta \le \delta} \Prob \left\{ \Vert X_n((\tau_n+\theta)\land T)-X_n(\tau_n)\Vert_E\ge \eta \right\}\le \varepsilon.
	\end{align*}
	In this case, we say that $(X_n)_{n\in\N_0}$
	satisfies the \emph{Aldous condition}  in $E$.
\end{Definition}

\begin{Prop}\label{AldousConditionTrunction}
	 The sequence $\big((u_n,F(u_n))\big)_{n\in\N_0}$ satisfies the Aldous condition  in $X_\gamma\times \LalphaPlusEinsDual.$
\end{Prop}

\begin{proof}
	As in \cite{ExistencePaper}, one can show that there is $C_1>0$ which satisfies 
	\begin{align*}
	\big(\E \big[\Vert u_n((\tau_n+\theta)\land T)-u_n(\tau_n)\Vert_\EAdual^2\big]\big)^{\nicefrac{1}{2}}
	\le C_1 \theta^{\nicefrac{1}{2}}
	\end{align*}
	for all $n\in\N_0$.
	Proposition~\ref{EstimatesGalerkinSolutionDefocusing} and Proposition~\ref{EstimatesGalerkinSolutionFocusing} imply that there is $C_2>0$ such that 
	\begin{align*}
	\big(\E \big[\Vert u_n((\tau_n+\theta)\land T)-u_n(\tau_n)\Vert_\EA^2\big]\big)^{\nicefrac{1}{2}}
	\le 2 \big(\E \big[\sup_{t\in [0,T]}\Vert u_n(t)\Vert_\EA^2\big]\big)^{\nicefrac{1}{2}}\le C_2
	\end{align*}	
	for all $n\in\N_0$. Using this and the fact that $X_\gamma=[\EAdual,\EA]_{\gamma+1/2}$, we obtain
	\begin{align*}
	&\big(\E \big[\Vert u_n((\tau_n+\theta)\land T)-u_n(\tau_n)\Vert_{X_\gamma}^2\big]\big)^{\nicefrac{1}{2}}
	\\&\le \big(\E \big[\Vert u_n((\tau_n+\theta)\land T)-u_n(\tau_n)\Vert_\EAdual^2\big]\big)^{\nicefrac{1}{2}(1/2-\gamma)} \big(\E \big[\Vert u_n((\tau_n+\theta)\land T)-u_n(\tau_n)\Vert_\EA^2\big]\big)^{\nicefrac{1}{2}(1/2+\gamma)}
	\\&\le C_1^{1/2-\gamma}\, C_2^{1/2+\gamma}\, \theta^{\nicefrac{1}{2}(1/2-\gamma)}
	\end{align*}	
	and the embedding $X_\gamma\hookrightarrow \LalphaPlusEins$ yields
	\begin{align*}
		&\E \big[\Vert F(u_n((\tau_n+\theta)\land T))-F(u_n(\tau_n))\Vert_{L^{\nicefrac{(\alpha+1)}{\alpha}}}\big]
		\\&\lesssim \E\big[\big(\norm{u_n((\tau_n+\theta)\land T)}_{L^{\alpha+1}}^{\alpha-1}+\norm{u_n(\tau_n)}_{L^{\alpha+1}}^{\alpha-1}\big) \Vert u_n((\tau_n+\theta)\land T))-u_n(\tau_n)\Vert_{L^{\alpha+1}}\big]
		\\&\le \big(\E\big[\big(\norm{u_n((\tau_n+\theta)\land T)}_{L^{\alpha+1}}^{\alpha-1}+\norm{u_n(\tau_n)}_{L^{\alpha+1}}^{\alpha-1}\big)^2\big]\big)^{\nicefrac{1}{2}} 
		\big(\E\big[\Vert u_n((\tau_n+\theta)\land T))-u_n(\tau_n)\Vert_{L^{\alpha+1}}^2\big]\big)^{\nicefrac{1}{2}} 
		\\&\le 2 \Big(\E\Big[\sup_{t\in [0,T]}\norm{u_n(t)}_{\EA}^{2(\alpha-1)}\Big]\Big)^{\nicefrac{1}{2}} 
		\big(\E\big[\Vert u_n((\tau_n+\theta)\land T))-u_n(\tau_n)\Vert_{X_\gamma}^2\big]\big)^{\nicefrac{1}{2}}\lesssim \theta^{\nicefrac{1}{2}(1/2-\gamma)}. 
	\end{align*}
	By the Tschebyscheff inequality, we obtain for all $\delta, \eta>0$
	\begin{align*}
	&\sup_{n\in\N_0} \sup_{0<\theta \le \delta}\Prob \left\{\Vert u_n((\tau_n+\theta)\land T)-u_n(\tau_n)\Vert_{X_\gamma}\ge \eta\right\}
	\\&\le \frac{1}{\eta^2} \sup_{n\in\N_0} \sup_{0<\theta \le \delta} \E \big[\Vert u_n((\tau_n+\theta)\land T)-u_n(\tau_n)\Vert_{X_\gamma}^2\big]
	\lesssim \frac{\delta^{1/2-\gamma}}{\eta^2}
	\end{align*}
	and 
		\begin{align*}
		&\sup_{n\in\N_0} \sup_{0<\theta \le \delta}\Prob \left\{\Vert F(u_n((\tau_n+\theta)\land T))-F(u_n(\tau_n))\Vert_{L^{\nicefrac{(\alpha+1)}{\alpha}}}\ge \eta\right\}
		\\&\le \frac{1}{\eta} \sup_{n\in\N_0} \sup_{0<\theta \le \delta} \E \big[\Vert F(u_n((\tau_n+\theta)\land T))-F(u_n(\tau_n))\Vert_{L^{\nicefrac{(\alpha+1)}{\alpha}}}\big]
		\lesssim \frac{\delta^{\nicefrac{1}{2}(1/2-\gamma)}}{\eta}.
		\end{align*}
	Therefore, we obtain that for all $\varepsilon,\eta>0$ there exists $\delta>0$ such that
	\begin{align*}
		&\sup_{n\in\N_0} \sup_{0<\theta \le \delta}\Prob \left\{\big\Vert \big(u_n((\tau_n+\theta)\land T),F(u_n((\tau_n+\theta)\land T))\big)-\big(u_n(\tau_n),F(u_n(\tau_n))\big)\big\Vert_{X_\gamma\times L^{\nicefrac{(\alpha+1)}{\alpha}}}\ge \eta\right\}
		\\&\le \sup_{n\in\N_0} \sup_{0<\theta \le \delta}\Prob \left\{\Vert u_n((\tau_n+\theta)\land T)-u_n(\tau_n)\Vert_{X_\gamma}\ge \eta\right\}
		\\&\quad+ \sup_{n\in\N_0} \sup_{0<\theta \le \delta}\Prob \left\{\Vert F(u_n((\tau_n+\theta)\land T))-F(u_n(\tau_n))\Vert_{L^{\nicefrac{(\alpha+1)}{\alpha}}}\ge \eta\right\}
		\\&\le \varepsilon.
	\end{align*}	
	This completes the proof Proposition~\ref{AldousConditionTrunction}.
\end{proof}

%% file: Compactness.tex
\section{About compactness in an appropriate locally convex space}\label{CompactnessSection}

In the previous section, we have shown that the solutions to the truncated equations satisfy uniform estimates in $\EA$ as well as the stochastic equicontinuity condition called Aldous condition. Now we will study a locally convex space in which these properties lead to tightness and whose topology is strong enough to pass to the limit in equation \eqref{galerkinEquation}. 


\begin{Definition}\label{DefinitionWeakTopologySpaces}
	Let $X$ be a Banach space with separable dual $X^*.$ 
Then, we define 
		\begin{align*}
		C_w([0,T],X):=\left\{ u: [0,T]\to X\colon [0,T]\ni t\to \duality{u(t)}{x^*}\in \C \text{ is cont. for all $x^*\in X^*$} \right\}.
		\end{align*}

\end{Definition}

Let $U$ be a separable Hilbert space which is compactly embedded and dense in $\EA$ (cf., e.g., Lemma C.1 in \cite{BrzezniakMotyl}). We point out that by the embeddings  $\LalphaPlusEinsDual\hookrightarrow \EAdual$ and $X_\gamma\hookrightarrow \EAdual$, we can deduce that also the embeddings $\LalphaPlusEinsDual\hookrightarrow U^*$ and $X_\gamma\hookrightarrow U^*$ are compact and dense.
Furthermore, we set
\begin{align}
Z= C_w([0,T],\EA\times \LalphaPlusEinsDual)\cap C([0,T],X_\gamma\times \LalphaPlusEinsDual)
\end{align}
and define a locally convex topology $\mathcal{Z}$ on $Z$ by the seminorms
\begin{align}\label{seminorms}
P=\{p_U\}\cup \{p_{h}\colon  h=(h_1,h_2)\in \EAdual\times \LalphaPlusEins\}
\end{align}
which satisfy for all $(u,G)\in Z$,  $h=(h_1,h_2)\in \EAdual\times \LalphaPlusEins$ that
\begin{align}\label{seminormU}
p_U(u,G)=\sup_{t\in [0,T]}\max\big\{\norm{u(t)}_{U^*},\norm{G}_{U^*}\big\}
\end{align}
and 
\begin{align}\label{seminormsWeakConvergence}
p_{h}(u,G)=\sup_{t\in [0,T]}\vert \duality{u(t)}{h_1}_{\EA,\EAdual}+ \duality{G(t)}{h_2}_{L^{(\alpha+1)/\alpha},L^{\alpha+1}}\vert.
\end{align}

We proceed with some properties of a function space which is closely connected to our truncated problem \ref{galerkinEquation}.
\begin{Lemma}\label{Lemma:galerkinSpaceClosed} Let $n\in\N$.
	\begin{enumerate}[a)]
		\item We have
		\begin{align*}
			C([0,T],H_n\times \LalphaPlusEinsDual)=\big\{(u,G)\in Z\colon \forall t\in [0,T]\colon u(t)=P_nu(t) \big\}.
		\end{align*}
		\item $C([0,T],H_n\times \LalphaPlusEinsDual)$ is a closed subset of $(Z,\mathcal{Z})$.
	\end{enumerate}
\end{Lemma}

\begin{proof}
	\emph{Step 1.} First, we show that $H_n=\{v\in H\colon v=P_n v\}$. The inclusion $\supset$ is an immediate consequence of $H_n=P_n(H)$. Let $v\in H_n$. Then, there exists $w\in H$ such that $v=P_n w$. Hence, we get $P_n v=P_n^2 w=P_n w=v$. This proves the inclusion $\subset$. 
	
	\emph{Step 2.} From the first step, we infer 
	\begin{align}\label{galerkinClosed:Inclusion}
		\big\{(u,G)\in Z\colon \forall t\in [0,T]\colon u(t)=P_nu(t) \big\}
		=\big\{(u,G)\in Z\colon \forall t\in [0,T]\colon u(t)\in H_n \big\}:=A
	\end{align}
	Furthermore, we observe $C([0,T],H_n\times \LalphaPlusEinsDual)\subset A$ due to the fact that $H_n\subset \EA$ by Lemma~\ref{PnPropertiesLzwei}. The inclusion $A\subset C([0,T],H_n\times \LalphaPlusEinsDual)$ is an immediate consequence of \eqref{galerkinClosed:Inclusion}.
	
	\emph{Step 3.} Let $\big((u_k,G_k)\big)_{k\in\N}\subset C([0,T],H_n\times \LalphaPlusEins)$ and $(u,G)\in Z$ satisfy that $(u_k,G_k)\to (u,G)$ in $Z$ as $k\to\infty$. 
	In particular, we obtain that $u\in C([0,T],H\times \LalphaPlusEins)$ and $u_k(t)\rightharpoonup u(t)$ in $\EA$ for all $t\in [0,T]$. 
	As a consequence of item~a) and  $P_n\in \LinearOperators{\EA}$ (cf. Lemma~\ref{PnPropertiesLzwei}), we deduce $u_k(t)=P_n u_k(t)\rightharpoonup P_n u(t)$ in $\EA$ for all $t\in [0,T]$. This implies $u(t)=P_nu(t)$ for all $t\in [0,T]$ and thus, by a), $(u,G)\in C([0,T],H_n\times \LalphaPlusEinsDual)$. The proof of item~b) is therefore completed.
\end{proof}

We next want to ensure that the quantities which occur in the uniform estimates from the previous section actually make sense as functions on $Z$. 

\begin{Lemma}\label{Lemma:LSC}
	\begin{enumerate}[(a)]
		\item For all $(u,G)\in Z$ it holds that 
		\begin{align*}
			\sup_{t\in [0,T]}\norm{u(t)}_{\EA}<\infty \qquad \text{and}\qquad \sup_{t\in [0,T]}\norm{G(t)}_{\LalphaPlusEinsDualNorm}<\infty.
		\end{align*}
		\item Let  $\Phi_j\colon Z\to [0,\infty)$, $j\in\{1,2\}$, be the functions which satisfy for all $(u,G)\in Z$ that 
		\begin{align*}
		\Phi_1(u,G)=\sup_{t\in [0,T]}\norm{u(t)}_{\EA}\qquad \text{and}\qquad \Phi_2(u,G)=\sup_{t\in [0,T]}\norm{G(t)}_{\LalphaPlusEinsDualNorm}.
		\end{align*}
		Then, it holds that $\Phi_1$ and $\Phi_2$ are lower semicontinuous.
	\end{enumerate}	
\end{Lemma}

\begin{proof}
	\emph{Step 1.} For item a), it is sufficient to prove that for a Banach space $E$ and $x\in C_w([0,T],E)$ it holds that $\sup_{t\in [0,T]}\norm{x(t)}_E<\infty.$ 
	From $x\in C_w([0,T],E)$, we infer that for all $x^*\in E^*$, we have $\sup_{t\in [0,T]} \vert \duality{x(t)}{x^*}\vert<\infty$. The uniform boundedness principle thus proves 
	\begin{align*}
		\sup_{t\in [0,T]}\norm{x(t)}_E=\sup_{t\in [0,T]} \sup_{\norm{x^*}_{E^*}\le 1}\vert \duality{x(t)}{x^*}\vert<\infty.
	\end{align*}
	\emph{Step 2.} By the dense embeddings $U\hookrightarrow \EA\hookrightarrow \EAdual$ and $U\hookrightarrow \EA\hookrightarrow \LalphaPlusEinsDual$ there exist sequences $(v_n)_{n\in\N}, (w_n)_{n\in\N}\subset U$  such that $(v_n)_{n\in\N}$ is dense in $\{v\in \EAdual\colon \norm{v}_\EAdual\le 1 \}$ and $(w_n)_{n\in\N}$ is dense in $\{w\in \LalphaPlusEins\colon \norm{w}_\LalphaPlusEins\le 1 \}$. This yields for all $(u,G)\in Z$
		\begin{align*}
		\Phi_1(u,G)=\sup_{t\in [0,T]}\sup_{n\in\N} \big\vert \duality{u(t)}{v_n}_{\EA,\EAdual}\big\vert=\sup_{n\in\N} \sup_{t\in [0,T]}\big\vert \duality{u(t)}{v_n}_{U^*,U}\big\vert. 
		\end{align*}
				\begin{align*}
				\text{and}\qquad \Phi_2(u,G)=\sup_{t\in [0,T]}\sup_{n\in\N} \big\vert \duality{G(t)}{w_n}_{L^{(\alpha+1)/\alpha},L^{\alpha+1}}\big\vert=\sup_{n\in\N} \sup_{t\in [0,T]}\big\vert \duality{G(t)}{w_n}_{U^*,U}\big\vert. 
				\end{align*}
	Now let $ (u_k,G_k) \to (u,G)$ in $Z$ as $k\to\infty$. Then, we deduce for each $n\in\N$ 
	\begin{align*}
		\sup_{t\in [0,T]}\big\vert \duality{u_k(t)-u(t)}{v_n}_{U^*,U}\big\vert \le \norm{u_k-u}_{C([0,T],U^*)} \norm{v_n}_U\to 0
	\end{align*}
	\begin{align*}
	\text{and}\qquad \sup_{t\in [0,T]}\big\vert \duality{G_k(t)-G(t)}{w_n}_{U^*,U}\big\vert \le \norm{G_k-G}_{C([0,T],U^*)} \norm{w_n}_U\to 0
	\end{align*}
	as $k\to \infty$. Combining this with the fact that the supremum of continuous functions is lower semicontinuous proves the assertion.
\end{proof}

\begin{Prop}\label{CompactnessDeterministic}
	Let $K$ be a subset of $Z$ and $r>0$ such that
	\begin{enumerate}
		\item[a)] $
		\sup_{(u,G)\in K} \sup_{t\in [0,T]} \max\Big\{ \Vert u(t)\Vert_{\EA}, \Vert G(t)\Vert_{\LalphaPlusEinsDualNorm}\Big\}\le r ;
		$
		\item[b)] $K$ is equicontinuous in $C([0,T],X_\gamma\times \LalphaPlusEinsDual),$ i.e.
		\begin{align*}
		\lim_{\delta \to 0} \sup_{(u,G)\in K} \sup_{\vert t-s\vert\le \delta} \max\Big\{\Vert u(t)-u(s)\Vert_{X_\gamma},\Vert G(t)-G(s)\Vert_{\LalphaPlusEinsDualNorm}\Big\}=0.
		\end{align*}
	\end{enumerate}
	Then, $K$ is sequentially relatively compact in $(Z,\mathcal{Z}).$
\end{Prop}

\begin{proof}

	\emph{Step 1:}
	Let $K$ be a subset of $Z$ such that the assumptions $a)$ and $b)$ are fulfilled. 
	Let us choose a sequence $\left((z_n,G_n)\right)_{n\in\N}\subset K.$ 
	
	Let $\left(t_j\right)_{j\in\N}\subset [0,T]\setminus I$ be a sequence that is dense in $[0,T].$ 	
	From the compactness of the embedding $U\hookrightarrow \EA $ and the continuity of the embeddings $\LalphaPlusEinsDual \hookrightarrow\EAdual$ and $X_\gamma\hookrightarrow \EAdual$ we infer that the embeddings $X_\gamma\hookrightarrow U^*$ and $\LalphaPlusEinsDual\hookrightarrow U^*$ are compact. 
	Combining this with Assumption $a)$, we can choose for each $j\in\N$ a  subsequence of $\left((z_n(t_j),G_n(t_j))\right)_{n\in\N}$  again denoted by $\left((z_n(t_j),G_n(t_j))\right)_{n\in\N}$ which converges strongly in $U^*\times U^*$ and weakly in $\EA\times \LalphaPlusEinsDual$. By a diagonalisation argument, one obtains a common subsequence $\left((z_n(t_j),G_n(t_j))\right)_{n\in\N}$ which converges strongly in $U^*\times U^*$ and weakly in $\EA\times \LalphaPlusEinsDual$.\\
	Let $\varepsilon>0$, $h_1\in \EAdual$,  and $h_2\in \LalphaPlusEins$. By density, we can choose $g_1\in X_{-\gamma}$ with $\Vert h_1-g_1 \Vert_{\EAdual}\le \varepsilon/(4r)$.
	Assumption $b)$ yields $\delta>0$ with
	\begin{align}\label{EquicontArzelaAscoli}
	&\sup_{(u,G)\in K} \sup_{\vert t-s\vert\le \delta} \max\Big\{\Vert u(t)-u(s)\Vert_{X_\gamma},\Vert G(t)-G(s)\Vert_{\LalphaPlusEinsDualNorm}\Big\} \nonumber\\&\le\frac{\varepsilon}{6}\min\big\{1,(\Vert{g_1}\Vert_{X_{-\gamma}}+\Vert{h_2}\Vert_{\LalphaPlusEinsKurz})^{-1}\big\}.
	\end{align}
	Let us choose finitely many open balls $U_\delta^1,\dots, U_\delta^L$ of radius $\delta$ covering $[0,T].$ By density, each of these balls contains an element of the sequence $\left(t_j\right)_{j\in\N},$ say $t_{j_l}\in U_\delta^l$ for $l\in \left\{1,\dots, L\right\}.$ In particular, the sequence $\left((z_{n}(t_{j_l}),G_n(t_{j_l}))\right)_{n\in\N}$ is strongly Cauchy in $U^*\times U^*$ and weakly Cauchy in $\EA\times \LalphaPlusEinsDual$ for all $l\in \left\{1,\dots, L\right\}.$
	Hence, there is $n_0\in\N$ such that for $n,m\in \N$ with $n,m\ge n_0$ we have
	\begin{align}\label{CauchyArzelaAscoli}
	\forall l=1,\dots,L\colon\quad \max\Big\{\Vert z_n(t_{j_l})-z_m(t_{j_l})\Vert_{U^*},\Vert G_n(t_{j_l})-G_m(t_{j_l})\Vert_{U^*}\Big\}
	\le \frac{\varepsilon}{6}
	\end{align}
	and 
	\begin{align}\label{CauchyWeak}
		\forall l=1,\dots,L\colon\quad \vert\Real \duality{z_n(t_{j_l})-z_m(t_{j_l})}{h_1}+\Real \duality{G_n(t_{j_l})-G_m(t_{j_l})}{h_2}\vert
		\le \frac{\varepsilon}{6}.
	\end{align}
	Now, we fix $t\in[0,T]$ and take $l\in \{1,\dots, L\}$ with $\vert t_{j_l}-t\vert\le \delta.$ 
	We  use  \eqref{EquicontArzelaAscoli} and \eqref{CauchyArzelaAscoli} to get for $n,m\ge n_0$
	\begin{align*}
	&\max\Big\{\Vert z_n(t)-z_m(t)\Vert_{U^*},\Vert G_n(t)-G_m(t)\Vert_{U^*}\Big\}
		\nonumber\\&\le \max\Big\{\Vert z_n(t)-z_n(t_{j_l})\Vert_{U^*},\Vert G_n(t)-G_n(t_{j_l})\Vert_{U^*}\Big\} 
		\nonumber\\&\quad +\max\Big\{\Vert z_n(t_{j_l})-z_m(t_{j_l})\Vert_{U^*},\Vert G_n(t_{j_l})-G_m(t_{j_l})\Vert_{U^*}\Big\} 
		\nonumber\\&\quad +\max\Big\{\Vert z_m(t_{j_l})-z_m(t)\Vert_{U^*},\Vert G_m(t_{j_l})-G_m(t)\Vert_{U^*}\Big\} 
	\nonumber\\&\lesssim \max\Big\{\Vert z_n(t)-z_n(t_{j_l})\Vert_{X_\gamma},\Vert G_n(t)-G_n(t_{j_l})\Vert_{\LalphaPlusEinsDualNorm}\Big\} 
	\nonumber\\&\quad +\max\Big\{\Vert z_n(t_{j_l})-z_m(t_{j_l})\Vert_{U^*},\Vert G_n(t_{j_l})-G_m(t_{j_l})\Vert_{U^*}\Big\} 
	\nonumber\\&\quad +\max\Big\{\Vert z_m(t_{j_l})-z_m(t)\Vert_{X_\gamma},\Vert G_m(t_{j_l})-G_m(t)\Vert_{\LalphaPlusEinsDualNorm}\Big\} 
	\nonumber\\&\le \frac{\varepsilon}{6}+\frac{\varepsilon}{6}+\frac{\varepsilon}{6}=\frac{\varepsilon}{2}.
	\end{align*}
		The estimates \eqref{EquicontArzelaAscoli} and \eqref{CauchyWeak} further yield for $n,m\ge n_0$
	\begin{align*}
		&\big\vert \Real \duality{z_n(t)-z_m(t)}{g_1}+\Real \duality{G_n(t)-G_m(t)}{h_2}\big\vert
		\\&\le \big\vert \Real \duality{z_n(t)-z_n(t_{j_l})}{g_1}+\Real \duality{G_n(t)-G_n(t_{j_l})}{h_2}\big\vert
		\\&\quad+\big\vert \Real \duality{z_n(t_{j_l})-z_m(t_{j_l})}{g_1}+\Real \duality{G_n(t_{j_l})-G_m(t_{j_l})}{h_2}\big\vert
		\\&\quad+\big\vert \Real \duality{z_m(t_{j_l})-z_m(t)}{g_1}+\Real \duality{G_m(t_{j_l})-G_m(t)}{h_2}\big\vert
				\\&\le  \Vert z_n(t)-z_n(t_{j_l})\Vert_{X_\gamma} \Vert{g_1}\Vert_{X_{-\gamma}}+\Vert{G_n(t)-G_n(t_{j_l})}\Vert_{\LalphaPlusEinsDualNorm} \Vert{h_2}\Vert_{\LalphaPlusEinsKurz}
				\\&\quad+\big\vert \Real \duality{z_n(t_{j_l})-z_m(t_{j_l})}{g_1}+\Real \duality{G_n(t_{j_l})-G_m(t_{j_l})}{h_2}\big\vert
				\\&\quad+\Vert z_m(t_{j_l})-z_m(t)\Vert_{X_\gamma} \Vert{g_1}\Vert_{X_{-\gamma}}+\Vert{G_m(t_{j_l})-G_m(t)}\Vert_{\LalphaPlusEinsDualNorm} \Vert{h_2}\Vert_{\LalphaPlusEinsKurz}
								\\&\le \frac{\varepsilon}{6}\min\big\{1,(\Vert{g_1}\Vert_{X_{-\gamma}}+\Vert{h_2}\Vert_{\LalphaPlusEinsKurz})^{-1}\big\}  \big(\Vert{g_1}\Vert_{X_{-\gamma}}+\Vert{h_2}\Vert_{\LalphaPlusEinsKurz}\big)
								\\&\quad+\frac{\varepsilon}{6}
								+\frac{\varepsilon}{6}\min\big\{1,(\Vert{g_1}\Vert_{X_{-\gamma}}+\Vert{h_2}\Vert_{\LalphaPlusEinsKurz})^{-1}\big\}  \big(\Vert{g_1}\Vert_{X_{-\gamma}}+\Vert{h_2}\Vert_{\LalphaPlusEinsKurz}\big)
								\\&\le \frac{\varepsilon}{2}.
	\end{align*}
	We conclude 	for $n,m\ge n_0$
	\begin{align*}
		&\big\vert \Real \duality{z_n(t)-z_m(t)}{h_1}+\Real \duality{G_n(t)-G_m(t)}{h_2}\big\vert
		\nonumber\\&\le \big\vert \Real \duality{z_n(t)-z_m(t)}{g_1}+\Real \duality{G_n(t)-G_m(t)}{h_2}\big\vert+\big\vert \Real \duality{z_n(t)-z_m(t)}{h_1-g_1}\big\vert
		\nonumber\\&\le\frac{\varepsilon}{2}+2 \Vert{h_1-g_1}\Vert_\EAdual\, \sup_{n\in\N}\sup_{t\in [0,T]}\Vert z_n(t)\Vert_\EA 
		\nonumber\\&\le\frac{\varepsilon}{2}+\frac{\varepsilon}{4r}2 r\le \varepsilon.
	\end{align*}
	In conclusion, we deduce for $n,m\ge n_0$
		\begin{align*}\label{Arzela2}
		&p_U(z_n-z_m,G_n-G_m)=\sup_{t\in [0,T]}\max\Big\{\Vert z_n(t)-z_m(t)\Vert_{U^*},\Vert G_n(t)-G_m(t)\Vert_{U^*}\Big\}\lesssim \varepsilon
		\end{align*}
	and 
		\begin{align*}
		&p_{(h_1,h_2)}(z_n-z_m,G_n-G_m)=\sup_{t\in [0,T]}\big\vert \Real \duality{z_n(t)-z_m(t)}{h_1}+\Real \duality{G_n(t)-G_m(t)}{h_2}\big\vert\le \varepsilon.
		\end{align*}
	This means that $\left((z_n,G_n)\right)_{n\in\N}$ is a Cauchy sequence in $C([0,T],U^*\times U^*)$ and in $C_w([0,T],\EA\times \LalphaPlusEinsDual)$.
	Hence, we obtain $(z,G)\in C([0,T],U^*\times U^*)$   and $(w,\mathcal{G})\in C_w([0,T],\EA\times \LalphaPlusEinsDual)$ with
	\begin{align*}
		p_{U}(z_n-z,G_n-G)\to 0 \qquad \text{and}\qquad p_{h}(z_n-w,G_n-\mathcal{G})\to 0, \qquad \forall \,h\in \EAdual\times \LalphaPlusEins
	\end{align*}
	 as $n\to \infty$.
	
	\emph{Step 2:} It is now sufficient to show that $(w,\mathcal{G})\in C([0,T],X_\gamma\times \LalphaPlusEinsDual)$ and $(z,G)=(w,\mathcal{G})$. Steps 1 and 2 prove that there is a subsequence $\left((z_n,G_n)\right)_{n\in\N}$ such that for all $t\in [0,T]$ we have 
	$(z_n(t),G_n(t))\rightharpoonup(z(t),G(t))$ in $U^*\times U^*$ and $(z_n(t),G_n(t))\rightharpoonup(w(t),\mathcal{G}(t))$ in $\EA\times \LalphaPlusEinsDual$. 
	We infer $(w(t),\mathcal{G}(t))=(z(t),G(t))\in \EA\times \LalphaPlusEinsDual$ for all $t\in [0,T].$ Moreover, we deduce for all $t,s\in [0,T]$ that 
	\begin{align*}
		(z_n(t)-z_n(s),G_n(t)-G_n(s))\rightharpoonup(w(t)-w(s),\mathcal{G}(t)-\mathcal{G}(s))
	\end{align*}
	in $X_\gamma\times \LalphaPlusEinsDual$. Consequently, for all $t,s\in [0,T]$ with $\vert t-s\vert\le \delta$ we have 
	\begin{align*}
		&\max\Big\{\Vert w(t)-w(s)\Vert_{X_\gamma},\Vert \mathcal{G}(t)-\mathcal{G}(s)\Vert_{\LalphaPlusEinsDualNorm}\Big\}
		\\&\le \liminf_{n\to\infty}\Big[\max\Big\{\Vert z_n(t)-z_n(s)\Vert_{X_\gamma},\Vert G_n(t)-G_n(s)\Vert_{\LalphaPlusEinsDualNorm}\Big\}\Big]
		\\&\le \sup_{(v,F)\in K} \sup_{\vert t-s\vert\le \delta} \max\Big\{\Vert v(t)-v(s)\Vert_{X_\gamma},\Vert F(t)-F(s)\Vert_{\LalphaPlusEinsDualNorm}\Big\}\le \frac{\varepsilon}{6}.
	\end{align*}
	This yields the continuity of $t\mapsto (w(t),\mathcal{G}(t))$ in $X_\gamma\times \LalphaPlusEins$,
which finally proves the assertion.
\end{proof}

As a consequence of the deterministic compactness criterion from Proposition~\ref{CompactnessDeterministic}, we get the following tightness result.

\begin{Korollar}\label{TightnessCriterion}
	Let $\big((u_n,G_n)\big)_{n\in\N}$ be a sequence of random variables in $(Z,\mathcal{Z})$ satisfying the Aldous condition in $X_\gamma\times \LalphaPlusEinsDual$ and
	\begin{align*}
	\sup_{n\in\N} \E \left[\Vert (u_n,G_n)\Vert_{L^\infty(0,T;\EA\times \LalphaPlusEinsDualNorm)}^2\right] <\infty.
	\end{align*}
	Then the sequence of laws $\left({\Prob}^{(u_n,G_n)}\right)_{n\in\N}$ is tight in $(Z,\mathcal{Z}),$ i.e. for every $\varepsilon>0$ there is a sequentially compact set $K_\varepsilon\subset Z$ with
	\begin{align*}
	\forall \, n\in\N\colon\quad \Prob^{(u_n,G_n)}(K_\varepsilon)\ge 1- \varepsilon.
	\end{align*}
\end{Korollar}

\begin{proof}
	Let $\varepsilon>0$, let 
	\begin{align*}
	R_1:= \left(\frac{2}{\varepsilon} \sup_{n\in\N} \E \left[ \Vert (u_n,G_n)\Vert_{L^\infty(0,T;\EA\times \LalphaPlusEinsDualNorm)}^2\right]\right)^{\frac{1}{2}},
	\end{align*}
	and let 
	\begin{align*}
	B:= \left\{(u,G)\in L^\infty(0,T;\EA\times \LalphaPlusEinsDual): \Vert (u,G)\Vert_{L^\infty(0,T;\EA\times \LalphaPlusEinsDualNorm)}\le R_1 \right\}.
	\end{align*}
	Using the Tschebyscheff inequality, we obtain
	\begin{align*}
	\Prob\left\{ \Vert (u_n,G_n)\Vert_{L^\infty(0,T;\EA\times \LalphaPlusEinsDualNorm)}> R_1\right\}\le \frac{1}{R_1^2}\E \left[\Vert (u_n,G_n)\Vert_{L^\infty(0,T;\EA\times \LalphaPlusEinsDualNorm)}^2\right]\le \frac{\varepsilon}{2}.
	\end{align*}	
	By Lemma 4.4 in \cite{ExistencePaper}, one can use the Aldous condition  to get a Borel subset $A$ of $C([0,T],X_\gamma\times \LalphaPlusEinsDual)$ with
	$\inf_{n\in\N}\Prob^{(u_n,G_n)}\left(A\right)\ge 1-\frac{\varepsilon}{2}$ and
	\begin{align*}
	\lim_{\delta\to 0} \sup_{u\in A} \sup_{\vert t-s\vert\le \delta}\max\Big\{ \Vert u(t)-u(s)\Vert_{X_\gamma}, \Vert G(t)-G(s)\Vert_{\LalphaPlusEinsDualNorm}\Big\}=0.
	\end{align*}
	From the Strauss Lemma (cf., e.g., Lemma A.3 in \cite{ExistencePaper}) and $X_\gamma\hookrightarrow \EA,$ we infer
	\begin{align*}
	A\cap B\subset C([0,T],X_\gamma\times \LalphaPlusEinsDual)\cap C_w([0,T],\EA\times \LalphaPlusEinsDual)=Z.
	\end{align*}
	We define $K:= \overline{A\cap B}$ where the closure is understood in $Z.$ The set $K$ is compact in $Z$ by Proposition $\ref{CompactnessDeterministic}$ and we can estimate
	\begin{align*}
	\Prob^{(u_n,G_n)}(K)\ge \Prob^{(u_n,G_n)}\left(A\cap B\right)\ge \Prob^{(u_n,G_n)}\left(A\right)-\Prob^{(u_n,G_n)}\left( B^c\right)\ge 1-\frac{\varepsilon}{2}-\frac{\varepsilon}{2}=1-\varepsilon,\qquad n\in\N.
	\end{align*}
	This completes the proof of Corollary~\ref{TightnessCriterion}.
\end{proof}

%% file: Limit.tex
\section{Passing to the limit in the truncated equation}

Based on the results in the previous sections, we would like to pass to the limit in equation \eqref{galerkinEquation} in order to prove our main result.
Crucial for this strategyv is Jakubowski's version of the Prokhorov-Skorohod procedure in nonmetric spaces. We state this result in the form of \cite{OndrejatWave} and refer to  \cite{Jakubowski}  for the original source.  Starting from \cite{BrzezniakOndrejat}, this theorem has been frequently to prove existence theorems for SPDE.

\begin{Prop}[Skorohod-Jakubowski]\label{SkohorodJakubowski}\index{Skorohod-Jakubowski-Theorem}
	Let $\mathcal{X}$ be a topological space such that there is a sequence of continuous functions $f_m: \mathcal{X}\to \R$ that separates points of $\mathcal{X}.$ Let $\mathcal{A}$ be the $\sigma$-algebra  generated by $\left(f_m\right)_m$. Then, 
	\begin{enumerate}[(i)]
		\item every compact subset of $\mathcal{X}$ is metrizable,
		\item every Borel subset of a $\sigma$-compact set in $\mathcal{X}$ belongs to $\mathcal{A}$, 
		\item every probability measure supported by a $\sigma$-compact set in $\mathcal{X}$ has a unique Radon extension to the Borel $\sigma$-algebra on $\mathcal{X}$, and
		\item for every tight sequence $\left(\mu_n\right)_{n\in\N}$  of probability measures on $\left(\mathcal{X}, \mathcal{A}\right)$ there exist a subsequence $\left(\mu_{n_k}\right)_{k\in\N},$ a probability space $(\tilde{\Omega},\tilde{\Filtration},\tildeProb)$, and Borel measurable random variables $X_k, X\colon \tilde{\Omega}\to \mathcal{X}$, $k\in\N$, with $\tildeProb^{X_k}=\mu_{n_k}$ for $k\in\N$ and $X_k \to X$ $\tildeProb$-almost surely for $k\to \infty.$
	\end{enumerate}
\end{Prop}

As in the previous section, $U$ is a separable Hilbert space which is compactly embedded and dense in $\EA$, 
 we denote
\begin{align}
Z= C_w([0,T],\EA\times \LalphaPlusEinsDual)\cap C([0,T],X_\gamma\times \LalphaPlusEinsDual),
\end{align}
and we define the topology $\mathcal{Z}$ on $Z$ by the family of seminorms $P$ from \eqref{seminorms}.

		\begin{Prop} \label{PropAlmostSureConvergence}
			Let $\left(u_n\right)_{n\in\N}$ be the sequence of solutions to the Galerkin equation $\eqref{galerkinEquation}.$
Then, there are a subsequence $\left(u_{n_k}\right)_{k\in\N}$, a probability space $(\tilde{\Omega},\tilde{\F},\tilde{\Prob})$ and Borel measurable random variables
				\begin{align}
					\big((v_k,\tilde{F}_k)\big)_{k\in\N}, (v,\tilde{F}):\tilde{\Omega} \rightarrow Z
				\end{align}
				which satisfy that 
				\begin{align}
					\tilde{\Prob}^{(v_k,\tilde{F}_k)}=\Prob^{(u_{n_k},F(u_{n_k}))}\qquad \text{and}\qquad (v_k,\tilde{F}_k)\to (v,\tilde{F})\quad \text{$\tilde{\Prob}$-a.s.  in $Z$}.
				\end{align}
%
		\end{Prop}
\begin{proof}
	Let $(\psi_k)_{k\in\N}\subset U\times U$ be a sequence which is dense in the unit ball of $U\times U$ and $(t_n)_{n\in\N}\subset [0,T]$ be a sequence which is dense in $[0,T]$. 
	Then, we define $f_{n,k}\colon Z\to \R$ by 
	\begin{align*}
		f_{n,k}(u,G)=\Real \big\langle{(u(t_n), G(t_n))},{\psi_k}\big\rangle_{U^*\times U^*, U\times U}
	\end{align*}
	for $(u,G)\in Z$, $n,k\in\N$. These functions are continuous by the choice of seminorms on $Z$. 
	Let us take $(u_1,G_1), (u_2,G_2)\in Z$ with $f_{n,k}(u_1,G_1)=f_{n,k}(u_2,G_2)$. The density of $(\psi_k)_{k\in\N}$ yields for all $n\in\N$, $\psi \in U\times U$ with $\norm{\psi}_{U\times U}\le 1$ that 
	\begin{align*}
		\Real \big\langle{(u_1(t_n)-u_2(t_n), G_1(t_n)-G_2(t_n))},{\psi}\big\rangle_{U^*\times U^*, U\times U}=0.
	\end{align*}
	Hence, we obtain for all $n\in\N$ that 
	\begin{align*}
		\big\Vert(u_1(t_n)-u_2(t_n), G_1(t_n)-G_2(t_n))\big\Vert_{U^*\times U^*}=0.
	\end{align*}
	Since $(u_1,G_1)$ and $(u_2,G_2)$ are both contained in $C([0,T],U^*\times U^*)$, we infer $(u_1(t),G_1(t))=(u_2(t),G_2(t))$ for all $t\in [0,T]$. We have thus proved that the sequence $(f_{n,k})_{n,k\in\N}$ separates points of $Z$. In particular, item (i) in Proposition~\ref{SkohorodJakubowski} implies that sequential compactness and compactness coincide in $Z$. 	
	Combining this with the  Propositions~\ref{EstimatesGalerkinSolutionDefocusing}, \ref{EstimatesGalerkinSolutionFocusing}, and \ref{AldousConditionTrunction} and Corollary~\ref{TightnessCriterion} shows that the law of $(u_n,F(u_n))_{n\in\N}$ is tight in $(Z,\mathcal{Z})$.  Item (iv) in Proposition~\ref{SkohorodJakubowski} then proves the assertion.
\end{proof}


As a Corollary of the previous result and the uniform estimates for $u_n$, $n\in\N$, we get: 
		\begin{Prop} \label{NewAprioriBounds}
			\begin{enumerate}
				\item[a)] We have $(v_k,\tilde{F}_k) \in  C([0,T],H_{n_k}\times \LalphaPlusEinsDual)$ $\tilde{\Prob}$-a.s.\ and for all $r\in [1,\infty),$ there is
				$C>0$
				with
				\begin{align*}
				\sup_{k\in\N} \Etilde \Big[ \sup_{t\in [0,T]}\norm{v_k(t)}_\EA^r\Big]\le C\qquad \text{and}\qquad \sup_{k\in\N}\Etilde \Big[ \sup_{t\in [0,T]}\norm{\tilde{F}_k(t)}_{\LalphaPlusEinsDualNorm}^{r}\Big]\le C.
				\end{align*}
				\item[b)] For all $r\in [1,\infty),$ we have
				\begin{align*}
				\Etilde \left[ \norm{v}_\LinftyEA^r\right]\le C
				\qquad \text{and}\qquad \Etilde \Big[ \norm{\tilde{F}}_{L^\infty(0,T;\LalphaPlusEinsDualNorm)}^{r}\Big]\le C
				\end{align*}
				with the same constant $C>0$ as in $a).$
			\end{enumerate}
		\end{Prop}
\begin{proof}
	\emph{Step 1.}
	From Lemma~\ref{Lemma:galerkinSpaceClosed}, we infer that $C([0,T],H_{n_k}\times \LalphaPlusEinsDual)$ is a closed subset of $(Z,\mathcal{Z})$. Combining this with $\tilde{\Prob}^{(v_k,\tilde{F}_k)}=\Prob^{(u_{n_k},F(u_{n_k}))}$ (cf. Proposition~\ref{PropAlmostSureConvergence}) and the fact that $(u_{n_k},F(u_{n_k}))\in C([0,T],H_{n_k}\times \LalphaPlusEinsDual)$ $\Prob$-almost surely proves that $(v_k,\tilde{F}_k)\in C([0,T],H_{n_k}\times \LalphaPlusEinsDual)$ $\tilde{\Prob}$-almost surely.
	Let us recall from Propositions~\ref{EstimatesGalerkinSolutionDefocusing} and \ref{EstimatesGalerkinSolutionFocusing} that there exists $C>0$ which satisfies
	\begin{align*}
	\sup_{n\in\N}\E \Big[\sup_{t\in[0,T]} \norm{u_n(t)}_\EA^{r}\Big]\le C\qquad \text{and}\qquad \sup_{n\in\N}\E \Big[\sup_{t\in[0,T]} \norm{F(u_n(t))}_{\LalphaPlusEinsDualNorm}^{r}\Big]\le C.
	\end{align*}
	Next note that Lemma~\ref{Lemma:LSC} ensures that the functions $\Phi_j\colon Z\to [0,\infty)$, $j\in\{1,2\}$, given by 
		\begin{align*}
		\Phi_1(u,G)=\sup_{t\in [0,T]}\norm{u(t)}_{\EA}\qquad \text{and}\qquad \Phi_2(u,G)=\sup_{t\in [0,T]}\norm{G(t)}_{\LalphaPlusEinsDualNorm}
		\end{align*}
	for $(u,G)\in Z$ are lower semicontinuous. In particular $\Phi_1$ and $\Phi_2$ are measurable  and thus, the 
	 identity $\tilde{\Prob}^{(v_k,\tilde{F}_k)}=\Prob^{(u_{n_k},F(u_{n_k}))}$ for all $k\in\N$ yields
	\begin{align*}
		\sup_{k\in\N}\Etilde \Big[\sup_{t\in[0,T]} \norm{v_k(t)}_\EA^{r}\Big]=\sup_{k\in\N}\E \Big[\sup_{t\in[0,T]} \norm{u_{n_k}(t)}_\EA^{r}\Big]\le C
	\end{align*}
	\begin{align*}
		\sup_{k\in\N}\Etilde \Big[\sup_{t\in[0,T]} \norm{\tilde{F}_k(t)}_{\LalphaPlusEinsDualNorm}^{r}\Big]=\sup_{k\in\N}\E \Big[\sup_{t\in[0,T]} \norm{F(u_{n_k}(t))}_{\LalphaPlusEinsDualNorm}^{r}\Big]\le C.
	\end{align*}
	\emph{Step 2.}
	Recall that we have $(v_k,\tilde{F}_k)\to (v,\tilde{F})$ almost surely  in $Z$ as $k\to\infty$. In particular, we obtain $v_k\to v$ and $\tilde{F}_k\to \tilde{F}$ in $C([0,T],U^*)$ as $k\to\infty$.
	By a), we further conclude for all $r,p\in [1,\infty)$ that
	\begin{align*}
		\norm{v_k}_{L^r(\tilde{\Omega}, L^p(0,T;U^*))}^r\lesssim \sup_{k\in\N} \Etilde \left[ \norm{v_k}_\LinftyEA^r\right]<\infty
	\end{align*}
		\begin{align*}
		\text{and}\qquad \norm{\tilde{F}_k}_{L^r(\tilde{\Omega}, L^p(0,T;U^*))}^r\lesssim \sup_{k\in\N} \Etilde \Big[ \norm{\tilde{F}_k}_{L^\infty(0,T;\LalphaPlusEinsDualNorm)}^r\Big]<\infty.
		\end{align*} 
	Vitali yields $v_k\to v$ in $L^r(\tilde{\Omega}, L^p(0,T;U^*))$ and $\tilde{F}_k\to \tilde{F}$ in $L^r(\tilde{\Omega}, L^p(0,T;U^*))$ for all $r,p\in [1,\infty)$. By a) and the Banach-Alaoglu-Theorem, there exist $\hat{v}\in L^r(\tilde{\Omega}, L^\infty(0,T;\EA))$ and $\hat{F}\in L^r(\tilde{\Omega}, L^\infty(0,T;L^{(\alpha+1)/\alpha}(M)))$ such that
	\begin{align*}
		v_k\rightharpoonup^* \hat{v} \quad \text{in $L^r(\tilde{\Omega}, L^\infty(0,T;\EA))$\quad and}\quad \Etilde \left[ \norm{\hat{v}}_\LinftyEA^r\right]\le C,
	\end{align*}
	\begin{align*}
	\tilde{F}_k\rightharpoonup^* \hat{F} \quad \text{in $L^r(\tilde{\Omega}, L^\infty(0,T;L^{(\alpha+1)/\alpha}(M)))$\quad and}\quad \Etilde \Big[ \norm{\hat{F}}_{L^\infty(0,T;L^{(\alpha+1)/\alpha})}^r\Big]\le C.
	\end{align*}
	 By the uniqueness of weak-star limits, we deduce $v=\hat{v}$ and $\tilde{F}=\hat{F}$, which proves the assertion of item~b). 
\end{proof}

		\begin{Lemma}\label{LemmaConvergences}
			Let $(n_k)_{k\in\N}\subset \N$ be a sequence with $n_k\to \infty$ as $k\to \infty$ and
			let $(z,G),(z_k,G_k)\in Z$ satisfy for $k\in\N$ that 
			$(z_k,G_k) \to (z,G)$ in $Z$ for $k\to \infty$. Then, for $t\in[0,T]$ it holds that
			\begin{align}
				z_k(t)&\rightharpoonup z(t) \quad \text{in $\EA$}\label{LemmaConvergences:One}\\
				Az_k(t)&\rightharpoonup Az(t) \quad \text{in $\EAdual$}\label{LemmaConvergences:Two}\\
				\mu_{n_k}(z_k(t))&\rightharpoonup \mu(z(t)) \quad \text{in $H$}\label{LemmaConvergences:Three}\\
				P_{n_k} G_k(t)&\rightharpoonup G(t) \quad \text{in $\EAdual$}\label{LemmaConvergences:Four}
			\end{align}		
			as $k\to \infty$. 	
		\end{Lemma}	
		
\begin{proof}
	From the definition of the topology $\mathcal{Z}$ on $Z$, we infer for all $t\in [0,T]$
		\begin{align}\label{LemmaConvergences:ConvergencesZ}
		z_k(t)\rightharpoonup z(t)\quad \text{in $\EA$}\qquad \text{and}\qquad G_k(t)\rightharpoonup G(t) \quad \text{in $L^{(\alpha+1)/\alpha}(M)$}. 
		\end{align}
	This, $A\in \LinearOperatorsTwo{\EA}{\EAdual}$, and the fact that bounded linear operators are weakly continuous show \eqref{LemmaConvergences:One} and \eqref{LemmaConvergences:Two}.
	Using \eqref{LemmaConvergences:ConvergencesZ}, Lemma~\ref{convergenceProperty}, and item~\eqref{PnExtended} in Lemma~\ref{PnPropertiesLzwei}, we deduce for all $\psi\in \EA$
	\begin{align*}
		\vert\dualityReal{P_{n_k} G_k(t)-G(t)}{\psi}\vert&=\vert\dualityReal{ G_k(t)}{P_{n_k}\psi}-\dualityReal{ G(t)}{\psi}\vert
		\\&\le \vert\dualityReal{ G_k(t)-G(t)}{\psi}\vert+\vert\dualityReal{ G_k(t)}{P_{n_k}\psi-\psi}\vert
		\\&\le \vert\dualityReal{ G_k(t)-G(t)}{\psi}\vert+\sup_{k\in\N}\norm{G_k(t)}_{L^{(\alpha+1)/\alpha}} \norm{P_{n_k}\psi-\psi}_\EA
		\\&\to 0 \qquad \text{as $k\to\infty$}.
	\end{align*}
Lemma~\ref{convergenceProperty} proves for all $m\in \N$, $\varphi\in H$ that
\begin{align*}
	\Vert S_n B_m S_n^2 B_m (S_n-I) \varphi\Vert_H\le \Vert B_m\Vert_{\LinearOperators{H}}^2 \Vert(S_n-I) \varphi\Vert_H\to 0,
\end{align*}
\begin{align*}
\Vert S_n B_m (S_n^2-I) B_m \varphi\Vert_H &\le \Vert  B_m\Vert_{\LinearOperators{H}} \Vert(S_n+I)(S_n-I) B_m \varphi\Vert_H
\\&\le 2 \Vert  B_m\Vert_{\LinearOperators{H}} \Vert(S_n-I) B_m \varphi\Vert_H\to 0,
\end{align*}
\begin{align*}
\text{and}\qquad \Vert (S_n-I) B_m^2\varphi\Vert_H\to 0.
\end{align*}
Consequently, we obtain for each $m\in \N$, $\varphi\in H$ that
\begin{align*}
	&\Vert (S_n B_m S_n)^2\varphi- B_m^2\varphi\Vert_H
	\\&\le \Vert S_n B_m S_n^2 B_m (S_n-I) \varphi\Vert_H+ \Vert S_n B_m (S_n^2-I) B_m \varphi\Vert_H+ \Vert (S_n-I) B_m^2\varphi\Vert_H 
	\to 0.
\end{align*}
Next note that $\Vert(S_n B_m S_n)^2\varphi \Vert_H\le \Vert  B_m\Vert_{\LinearOperators{H}}^2 \norm{\varphi}_H$ for $m,n\in\N$, the hypothesis $\sumM \norm{B_m}_{\LinearOperators{H}}^2<\infty$ (cf. Assumption~\ref{stochasticAssumptions}), and Lebesgue's dominated convergence theorem ensure for all $\varphi\in H$ 
	\begin{align*}
	&\norm{\mu_{n}(\varphi)-\mu(\varphi)}_H
	\le \tfrac{1}{2}\sumM \big\Vert (S_n B_m S_n)^2\varphi- B_m^2\varphi\big\Vert_H	\to 0.
	\end{align*}
Moreover, observe that by $\sumM \norm{B_m}_{\LinearOperators{H}}^2<\infty$ (cf. Assumption~\ref{stochasticAssumptions}) and the fact that $S_n$, $n\in\N$, are contractive selfadjoint operators (cf. Lemma~\ref{PaleyLittlewoodLemma}), we conclude that $\mu_{n_k}$ and $\mu$ are bounded selfadjoint operators on $H$.
Thus, we get	 for all $\varphi\in H$ 
	\begin{align*}
		\vert \skpHReal{\mu_{n_k}(z_k(t))-\mu(z(t))}{\varphi}\vert&=\vert \skpHReal{z_k(t)}{\mu_{n_k}(\varphi)}-\skpHReal{z(t)}{\mu(\varphi)}\vert
		\\&\le \vert \skpHReal{z_k(t)-z(t)}{\mu(\varphi)}\vert+\vert \skpHReal{z_k(t)}{\mu_{n_k}(\varphi)-\mu(\varphi)}\vert
		\\&\le \vert \skpHReal{z_k(t)-z(t)}{\mu(\varphi)}\vert+\sup_{k\in\N}\norm{z_k(t)}_H\norm{\mu_{n_k}(\varphi)-\mu(\varphi)}_H
		\\&\to 0 \qquad \text{as $k\to\infty$}.
	\end{align*}
We have therefore proved \eqref{LemmaConvergences:Four}, which completes the proof of Lemma~\ref{LemmaConvergences}.
\end{proof}

	Consider the map $\mathcal{M}_k\colon Z\to C([0,T],\EAdual)$ which satisfies for all $u\in Z$, $t\in [0,T]$ that
	\begin{align*}
	\big(\mathcal{M}_k(u,G)\big)(t):=-u(t)+ S_{n_k} u_0&+ \int_0^t \left[-\im A u(s)-\im P_{n_k} G(s)+\mu_{n_k}(u(s))\right] \df s.
	\end{align*}
Note that $\mathcal{M}_k$ maps to $C([0,T],\EAdual)$ as a consequence of Lemma~\ref{Lemma:LSC}.
\begin{Lemma}\label{continuity:MartingaleMaps}
		Let $0\le s\le t\le T$, $\varphi_1, \varphi_2\in H$, $\psi\in \EA$. Let $\Phi_k\colon Z\to \R$, $k\in\N$, and $\Psi_k\colon Z\to \R$, $k\in\N$, be the functions which satisfy 
		\begin{align*}
		\Phi_k(u,G)=\Real \duality{\mathcal{M}_k(u,G)(t)}{\psi}
		\end{align*}
		\begin{align*}
		\text{and}\qquad\Psi_k(u,G)=\sumM \int_s^t  \skpHReal{\im S_{n_k} B_m S_{n_k} u(r)}{\varphi_1} \skpHReal{\im S_{n_k} B_m S_{n_k} u(r)}{\varphi_2} \df r
		\end{align*}
		for all $k\in\N$. Then, for all $k\in\N$ it holds that $\Phi_k$ and $\Psi_k$ are measurable.
\end{Lemma}

\begin{proof}
	\emph{ad $\Phi_k$:} For $R\in \N$, we define 
	\begin{align*}
		Z_R:=\Big\{(u,G)\in Z\colon \max\big\{	\sup_{t\in [0,T]}\norm{u(t)}_{\EA},\sup_{t\in [0,T]}\norm{G(t)}_{\LalphaPlusEinsDualNorm}\big\}  \le R \Big\}.
	\end{align*}
	Lemma~\ref{Lemma:LSC} then implies that $Z_R$ is a closed set in $Z$ and 
	\begin{align*}
		\Phi_k(u,G)=\lim_{R\to \infty} \mathbf{1}_{Z_R}(u,G) \Real \duality{\mathcal{M}_k(u,G)(t)}{\psi}.
	\end{align*}
	It is thus sufficient to show that the maps 
	\begin{align*}
		Z_R\ni (u,G)\mapsto \Real \duality{\mathcal{M}_k(u,G)(t)}{\psi}\in \R
	\end{align*}
	are continuous. In order to show this, we take $Z_R\ni (u_n,G_n)\to (u,G)\in Z_R$, $\psi\in \EA$ and consider 
		\begin{align}\label{martingaleMaps:One}
		&\vert\duality{\mathcal{M}_k(u_n,G_n)(t)}{\psi}-\duality{\mathcal{M}_k(u,G)(t)}{\psi}\vert
		\nonumber\\&\le  \big\vert \duality{u(t)-u_n(t)}{\psi}\big\vert+ \int_0^t \big\vert\duality{\im A\big[u(r)-u_n(r)\big]}{\psi}\big\vert\df r\nonumber\\&\qquad + \int_0^t\big\vert\duality{\im P_{n_k} \big[G(r)-G_n(r)\big]}{\psi}\big\vert\df r
		+\int_0^t\big\vert\duality{\mu_{n_k}[u_n(r)-u(r)\big]}{\psi}\big\vert  \df r.
		\end{align}
		Lemma~\ref{LemmaConvergences} yields that the integrands on the RHS of \eqref{martingaleMaps:One} convergence to $0$ for each $r\in [0,T].$ The definition of $Z_R$ and Lebesgue's dominated convergence theorem hence imply 
		\begin{align*}
			\vert\duality{\mathcal{M}_k(u_n,G_n)(t)}{\psi}-\duality{\mathcal{M}_k(u,G)(t)}{\psi}\vert\to 0
		\end{align*}
		as $n\to \infty$, which finally proves the assertion. \\
		\emph{ad $\Psi_k$:} For fixed $\varphi\in H$, we consider the linear operator $\mathfrak{T}\colon Z\to L^2([s,t]\times \N)$ given by 
			\begin{align*}
			\mathfrak{T}(u,G)= \big(\skpHReal{\im S_{n_k} B_m S_{n_k} u}{\varphi}\big)_{m\in\N}
			\end{align*}
			and take $Z_R\ni (u_n,G_n)\to (u,G)\in Z_R$. 
Since the operators $S_{n_k} B_m S_{n_k}$ are bounded and selfadjoint and $u_n(r)\rightharpoonup u(r)$ in $H$ for all $r\in [s,t]$, we infer that for fixed $m\in\N$, $r\in [s,t]$ it holds that
\begin{align*}
	\skpHReal{\im S_{n_k} B_m S_{n_k} u_n(r)}{\varphi}\to \skpHReal{\im S_{n_k} B_m S_{n_k} u(r)}{\varphi}
\end{align*}
as $n\to \infty$. Due to the estimate 
\begin{align*}
	\vert \skpHReal{\im S_{n_k} B_m S_{n_k} u_n(r)}{\varphi}\vert \le \norm{B_m}_{\LinearOperators{H}} \sup_{r\in[s,t]}\norm{u_n(r)}_H \norm{\varphi}_H\le R \norm{B_m}_{\LinearOperators{H}} \norm{\varphi}_H\in L^2([s,t]\times \N),
\end{align*}
Lebesgue's dominated convergence theorem yields 
\begin{align*}
\mathfrak{T}(u_n,G_n)=\Big(\skpHReal{\im S_{n_k} B_m S_{n_k} u_n}{\varphi}\Big)_{m\in\N}\to \Big(\skpHReal{\im S_{n_k} B_m S_{n_k} u}{\varphi}\Big)_{m\in\N}=\mathfrak{T}(u,G)
\end{align*}
in $L^2([s,t]\times \N)$. By the continuity of the inner product on $L^2([s,t]\times \N)$, we get
\begin{align*}
	\Psi_k(u_n,G_n)\to \Psi_k(u,G),\qquad n\to \infty.
\end{align*}
As in the case of $\Phi_k$ this implies that $\Psi_k\colon Z\to \R$ is measurable.
\end{proof}

				\begin{Lemma}
					Let $N_k: \tilde{\Omega} \times [0,T] \rightarrow H_{n_k}$ satisfy
					$N_k(\tilde{\omega},t)=\mathcal{M}_k((v_k,\tilde{F}_k)(\tilde{\omega},t))$
					for $k\in \N$ and $t\in[0,T]$. Then it holds that 
	\begin{align*}
	\Etilde \left[ \Real \duality{N_k(t)-N_k(s)}{\psi} h(v_k|_{[0,s]})\right]=0
	\end{align*}
	and
	\begin{align*}
	\Etilde \Bigg[& \Bigg(\Real \duality{N_k(t)}{\psi}\Real \duality{N_k(t)}{\varphi}-\Real \duality{N_k(s)}{\psi}\Real \duality{N_k(s)}{\varphi}\\
	&\hspace{1 cm}-\sumM \int_s^t  \Real \duality{\im S_{n_k} B_m S_{n_k} v_k(r)}{\psi} \Real \duality{\im S_{n_k} B_m S_{n_k} v_k(r)}{\varphi} \df r\Bigg) h(v_k|_{[0,s]})\Bigg]=0
	\end{align*}
	for all $0\le s\le t\le T$, $\psi, \varphi \in \EA$ and bounded, continuous functions $h$ on $C([0,T],U^*).$ 
					\end{Lemma}	
\begin{proof}
	Let  $M_k: \Omega \times [0,T] \rightarrow H_{n_k}$, $k\in\N$, be the  stochastic processes which satisfy
	\begin{align*}
		M_k(\omega,t)=\mathcal{M}_k((u_{n_k},F(u_{n_k}))(\omega,t)),\qquad \omega\in\Omega,\, t\in [0,T].
	\end{align*}
	Recall from Proposition~\ref{MassEstimateGalerkinSolution} and \eqref{galerkinEquation} that we have 
		\begin{align*}
		M_k(t)&=-u_{n_k}(t)+ S_{n_k} u_0+ \int_0^t \big[-\im A u_{n_k}(s)-\im P_{n_k} F\left( u_{n_k}(s)\right)+\mu_{n_k}(u_{n_k}(s))\big] \df s
		\\&= \im \int_0^t S_{n_k} B (S_{n_k} u_{n_k}(s)) \df W(s).
		\end{align*}
	Hence, we obtain that $M_k$ is a continuous square integrable martingale w.r.t the filtration $\F_{k,t}:=\sigma \left(u_{n_k}(s): s\le t\right)$. The quadratic variation is given by
	\begin{align*}
	\quadVar{M_k}_t\psi= \sumM \int_0^t \big[\im S_{n_k} B_m S_{n_k} u_{n_k}(s)\big] \skpHReal{\im S_{n_k} B_m S_{n_k} u_{n_k}(s)}{\psi} \df s,\qquad \psi \in {H}.
	\end{align*}
	From Lemma A.16 in Appendix A in \cite{dissertationFH}, we infer
				\begin{align*}
				\E \left[ \skpHReal{M_k(t)-M_k(s)}{\psi} h(u_{n_k}|_{[0,s]})\right]=0
				\end{align*}
				and
				\begin{align*}
				\E \Bigg[& \Bigg(\skpHReal{M_k(t)}{\psi}\skpHReal{M_k(t)}{\varphi}-\skpHReal{M_k(s)}{\psi}\skpHReal{M_k(s)}{\varphi}\\
				&\hspace{1 cm}-\sumM \int_s^t  \skpHReal{\im S_{n_k} B_m S_{n_k} u_{n_k}(r)}{\psi} \skpHReal{\im S_{n_k} B_m S_{n_k} u_{n_k}(r)}{\varphi} \df r\Bigg) h(u_{n_k}|_{[0,s]})\Bigg]=0
				\end{align*}
				for all $0\le s\le t\le T$, $\psi, \varphi \in {H}$ and bounded, continuous functions $h$ on $C([0,T],{H_{n_k}}).$
%
	By the identity $\tilde{\Prob}^{(v_k,\tilde{F}_k)}=\Prob^{(u_{n_k},F(u_{n_k}))}$ and Lemma~\ref{continuity:MartingaleMaps}, we deduce
	\begin{align}\label{approxMartingaleExpecApprox}
	\Etilde \left[ \Real \duality{N_k(t)-N_k(s)}{\psi} h(v_k|_{[0,s]})\right]=0
	\end{align}
	and
	\begin{align}\label{approxMartingaleVariationApprox}
	\Etilde \Bigg[& \Bigg(\Real \duality{N_k(t)}{\psi}\Real \duality{N_k(t)}{\varphi}-\Real \duality{N_k(s)}{\psi}\Real \duality{N_k(s)}{\varphi}\nonumber\\
	&\hspace{1 cm}-\sumM \int_s^t  \Real \duality{\im S_{n_k} B_m S_{n_k} v_k(r)}{\psi} \Real \duality{\im S_{n_k} B_m S_{n_k} v_k(r)}{\varphi} \df r\Bigg) h(v_k|_{[0,s]})\Bigg]=0
	\end{align}
	for all $0\le s\le t\le T$, $\psi, \varphi \in\EA$ and bounded, continuous functions $h$ on $C([0,T],U^*).$ 
\end{proof}

	Let $\iota: U \hookrightarrow {H}$ be the usual embedding, $\iota^*: {H} \rightarrow U$ its Hilbert-space-adjoint,
	i.e. $\skpH{\iota u}{v}=\skp{u}{\iota^* v}_U$ for $u\in U$ and $v\in{H}.$
	Further, we set $L:=\left(\iota^*\right)': U^* \rightarrow {H}$  as the dual operator of $\iota^*$ with respect to the Gelfand triple $U\hookrightarrow H\eqsim H^*\hookrightarrow U^*.$

		\begin{Lemma}\label{martingaleRepresentation}
			For  $(u,G)\in Z$, $t\in [0,T]$, we denote 
			\begin{align}\label{TildeM}
			\big(\mathcal{M}(u,G)\big)(t):=-u(t)+ u_0&+ \int_0^t \left[-\im A u(s)-\im  G(s)+\mu(u(s))\right] \df s
			\end{align}
						and	$N(\tilde{\omega},t)=\mathcal{M}((v,\tilde{F})(\tilde{\omega},t))$ for $\tilde{\omega}\in \tilde{\Omega}$, $t\in [0,T]$.
							Then, the process $LN: \tilde{\Omega} \times [0,T] \rightarrow H$  is a continuous square integrable martingale in $H$ 
							w.r.t the filtration $\tilde{\F}_{t}:=\sigma \left(v(s): s\le t\right).$ The quadratic variation of $LN$ is given by
							\begin{align*}
							\quadVar{L N}_t\zeta= \sumM \int_0^t \im L B_m v(s)\, \skpHReal{\im L B_m v(s)}{\zeta} \df s,\qquad \zeta \in {H}.
							\end{align*}
		\end{Lemma}

\begin{proof}
	Throughout this proof let $s\in [0,T]$ and let $h$ be a bounded, continuous function on $C([0,s],U^*)$.
	
	\emph{Step 1.} From Proposition~\ref{NewAprioriBounds}, we get $v\in L^\infty(0,T;\EA)$,  $\tilde{F}\in L^\infty(0,T;\LalphaPlusEinsDual)$, and $\mu(v)\in L^\infty(0,T;\EA)$ almost surely.
	The fact that $v\in Z$ almost surely further shows $v\in C([0,T],\EAdual)$ almost surely. The fact that $N=\mathcal{M}((v,\tilde{F}))$ and \eqref{TildeM} thus show that $N$ has continuous paths in $\EAdual$. The process $LN$ therefore has continuous paths in $H$.
	
	\emph{Step 2.} Recall from Proposition~\ref{PropAlmostSureConvergence} that we have $(v_k,\tilde{F}_k)\to (v,\tilde{F})$ almost surely in $Z$ as $k\to\infty$. In particular, this yields $v_k\to v$ almost surely in $C([0,T],U^*)$. From Lemma~\ref{LemmaConvergences}, we can thus deduce
	 that for all $\psi\in \EA$ it holds almost surely for all $r\in [0,T]$  that 
	\begin{align*}
		\duality{-\im A v_k(r)-\im  P_{n_k}\tilde{F}_k(r)+\mu_{n_k}(v_k(r))}{\psi} h(v_k|_{[0,s]}) \to \duality{-\im A v(r)-\im  \tilde{F}(r)+\mu(v_k(r))}{\psi} h(v|_{[0,s]})
	\end{align*}
	\begin{align*}
		 \text{and}\qquad\duality{v_k(t)}{\psi}h(v_k|_{[0,s]})\to \duality{v(t)}{\psi}h(v|_{[0,s]})
	\end{align*}
	as $k\to\infty$.
From Proposition~\ref{NewAprioriBounds}, we further obtain  
	\begin{align*}
&\Vert\duality{-\im A v_k-\im  P_{n_k}\tilde{F}_k+\mu_{n_k}(v_k)}{\psi} h(v_k|_{[0,s]})\Vert_{L^p(\tilde{\Omega}\times [0,T])}
		\\&\le \norm{\psi}_\EA \big[\sup\big\{\vert h(f)\vert\colon f\in C([0,s],U^*) \big\}\big]  \\&\quad\cdot\Big\Vert\norm{Av_k}_\EAdual+\norm{P_{n_k}\tilde{F}_k}_{\EAdual}+\norm{\mu_{n_k}(v_k)}_H\Big\Vert_{L^p(\tilde{\Omega}\times [0,T])}
		\\&\lesssim \norm{\psi}_\EA \big[\sup\big\{\vert h(f)\vert\colon f\in C([0,s],U^*) \big\}\big] 
		\\&\quad\cdot \sup_{k\in\N}\Big[\norm{v_k}_{L^p(\tilde{\Omega}\times [0,T],\EA)}+\norm{\tilde{F}_k}_{L^p(\tilde{\Omega}\times [0,T],\LalphaPlusEinsDualNorm)}\Big]
		<\infty
	\end{align*}
	for all $p\in [1,\infty)$, $\psi\in \EA$  as well as
	\begin{align*}
		\Vert\duality{v_k(t)}{\psi}h(v_k|_{[0,s]})\Vert_{L^p(\tilde{\Omega})}\le \norm{\psi}_\EA \big[\sup\big\{\vert h(f)\vert\colon f\in C([0,s],U^*) \big\}\big] \sup_{k\in\N}\norm{v_k(t)}_{L^p(\tilde{\Omega},\EAdual)}<\infty
	\end{align*}
		for all $p\in [1,\infty)$, $t\in [0,T]$, $\psi\in \EA$.
	Vitali's convergence theorem thus proves
		\begin{align*}
		\duality{-\im A v_k-\im  P_{n_k}\tilde{F}_k+\mu_{n_k}(v_k)}{\psi} h(v_k|_{[0,s]}) \to \duality{-\im A v-\im  \tilde{F}+\mu(v_k)}{\psi} h(v|_{[0,s]})
		\end{align*}
	in $L^p(\tilde{\Omega}\times [0,T])$	as $k\to\infty$ for all $p\in [1,\infty)$, $\psi\in \EA$ and
		\begin{align*}
	\duality{v_k(t)}{\psi}h(v_k|_{[0,s]})\to \duality{v(t)}{\psi}h(v|_{[0,s]})
		\end{align*}
	in $L^p(\tilde{\Omega})$  for all $p\in [1,\infty)$, $t\in [0,T]$, $\psi\in \EA$.
Hence, 
	\begin{align*}
		&\big\Vert\Real\duality{N(t)}{\psi} h(v|_{[0,s]})-\skpHReal{N_k(t)}{\psi} h(v_k|_{[0,s]})\big\Vert_{L^p(\tilde{\Omega})}
		\\&\le \Etilde \bigg(\bigg\vert\int_0^t \duality{-\im A v_k(r)-\im  P_{n_k}\tilde{F}_k(r)+\mu_{n_k}(v_k(r))}{\psi} h(v_k|_{[0,s]})
		\\&\hspace{5cm}-\duality{-\im A v(r)-\im  \tilde{F}(r)+\mu(v_k(r))}{\psi} h(v|_{[0,s]})  \df r\bigg\vert^p\bigg)^{\!1/p}
				\\&\quad + \big\Vert \duality{v(t)}{\psi}h(v|_{[0,s]})-\duality{v_k(t)}{\psi}h(v_k|_{[0,s]})\big\Vert_{L^p(\tilde{\Omega})}
		\\&\lesssim \Big\Vert\duality{-\im A v_k-\im  P_{n_k}\tilde{F}_k+\mu_{n_k}(v_k)}{\psi} h(v_k|_{[0,s]}) - \duality{-\im A v-\im  \tilde{F}+\mu(v)}{\psi} h(v|_{[0,s]})\Big\Vert_{L^p(\tilde{\Omega}\times [0,T])}
						\\&\quad + \big\Vert \duality{v(t)}{\psi}h(v|_{[0,s]})-\duality{v_k(t)}{\psi}h(v_k|_{[0,s]})\big\Vert_{L^p(\tilde{\Omega})}
		\\&\to 0,\qquad k\to\infty
	\end{align*}	
	for all $p\in [1,\infty)$, $t\in [0,T]$, $\psi\in \EA$. Therefore, we obtain
		\begin{align}\label{martingaleLimitOne}
		\Etilde \left[ \Real\duality{N(t)-N(s)}{\psi} h(v|_{[0,s]})\right]=\lim_{k\to\infty}\Etilde \left[ \Real \duality{N_k(t)-N_k(s)}{\psi} h(v_k|_{[0,s]})\right]=0
		\end{align}
		and
		\begin{align}\label{martingaleLimitTwo}
			\Etilde \Big[& \big(\Real\duality{N(t)}{\psi}\Real\duality{N(t)}{\varphi}-\Real\duality{N(s)}{\psi}\Real\duality{N(s)}{\varphi}\big)  h(v|_{[0,s]}) \Big]
			\nonumber\\&=\lim_{k\to\infty} \Etilde \Big[ \big(\Real \duality{N_k(t)}{\psi}\Real \duality{N_k(t)}{\varphi}-\Real\duality{N_k(s)}{\psi}\Real\duality{N_k(s)}{\varphi}\big)  h(v_k|_{[0,s]}) \Big]
		\end{align}
	for all $t\in [s,T]$, $\psi,\varphi\in \EA$.

\emph{Step 3.} Note that as in the proof of Lemma~\ref{LemmaConvergences}, one can deduce 	
\begin{align*}
	\norm{S_{n_k} B_m S_{n_k}\psi-B_m \psi}_H\to 0
\end{align*}
as $k\to\infty$ for all $\psi\in H$. Since we have $v_k(t)\rightharpoonup v(t)$ in $H$ almost surely for all $t\in [0,T]$  by Lemma~\ref{LemmaConvergences}, we get
\begin{align*}
	&\big\vert \Real\duality{S_{n_k} B_m S_{n_k} v_k(t)-B_m v(t)}{\psi}\big\vert=\big\vert \Real\duality{ v_k(t)}{S_{n_k} B_m S_{n_k}\psi}-\Real\duality{v(t)}{B_m \psi} \big\vert
	\\&\le \big\vert \Real\duality{ v_k(t)}{S_{n_k} B_m S_{n_k}\psi-B_m \psi}\big\vert+\big\vert\Real\duality{v_k(t)-v(t)}{B_m \psi} \big\vert
	\\&\le \sup_{k\in\N} \norm{v_k(t)}_H \norm{S_{n_k} B_m S_{n_k}\psi-B_m \psi}_H+\big\vert\Real\duality{v_k(t)-v(t)}{B_m \psi} \big\vert\to 0
\end{align*}
almost surely for all $t\in [0,T]$, $m\in\N$  as $k\to\infty$.
Using the fact that $h(v_k|_{[0,s]})\to h(v|_{[0,s]})$ almost surely, the bound \begin{align*}
	&\vert\Real\duality{S_{n_k} B_m S_{n_k} v_k(t)}{\psi}\vert^2h(v_k|_{[0,s]}) 
	\\&\le \norm{\psi}_H^2 \norm{B_m}_{\LinearOperators{H}}^2  \sup_{k\in\N} \norm{v_k(t)}^2_H \big[\sup\big\{\vert h(f)\vert\colon f\in C([0,s],U^*) \big\}\big] \in \ell^1(\N)
\end{align*}
almost surely for all $t\in [0,T]$, and Lebesgue's dominated convergence theorem, we infer 
\begin{align*}
	\Real\duality{S_{n_k} B_m S_{n_k} v_k(t)}{\psi}\,h(v_k|_{[0,s]}) \to\Real\duality{B_m v(t)}{\psi}\, h(v|_{[0,s]}) 
\end{align*}
in $\ell^2(\N)$ almost  surely for all $t\in [0,T]$ as $k\to\infty$.
The estimate
\begin{multline*}
	\Big\Vert\Big(\sumM\big\vert \Real\duality{S_{n_k} B_m S_{n_k} v_k}{\psi}\,h(v_k|_{[0,s]})\big\vert^2\Big)^{\!1/2}\Big\Vert_{L^p(\tilde{\Omega}\times [0,T])}
	\\\le \norm{\psi}_H \Big(\sumM\norm{B_m}_{\LinearOperators{H}}^2\Big)^{\!1/2} \sup_{k\in\N}\norm{v_k}_{L^p(\tilde{\Omega}\times [0,T],H)}\big[\sup\big\{\vert h(f)\vert\colon f\in C([0,s],U^*) \big\}\big]<\infty
\end{multline*}
and Vitali's convergence theorem thus ensure for all $p\in [1,\infty)$ that 
\begin{align*}
\Real\duality{S_{n_k} B_m S_{n_k} v_k}{\psi}\,h(v_k|_{[0,s]}) \to\Real\duality{B_m v}{\psi}\, h(v|_{[0,s]})  \qquad \text{in $L^p(\tilde{\Omega}\times [0,T],\ell^2(\N))$ as $k\to\infty$.}
\end{align*}
Consequently,
\begin{align}\label{martingaleLimitThree}
&\lim_{k\to\infty}	\Etilde \bigg[\sumM \int_s^t  \skpHReal{\im S_{n_k} B_m S_{n_k} v_k(r)}{\psi} \skpHReal{\im S_{n_k} B_m S_{n_k} v_k(r)}{\varphi} \df r\,  h(v_k|_{[0,s]})\bigg]
\nonumber\\&=\lim_{k\to\infty} \Big(\skpHReal{\im S_{n_k} B_m S_{n_k} v_k(r)}{\psi}h(v_k|_{[0,s]}),\skpHReal{\im S_{n_k} B_m S_{n_k} v_k(r)}{\varphi}\Big)_{L^2(\tilde{\Omega}\times [s,t]\times \N)}
\nonumber\\&=\Big(\Real\duality{\im  B_m  v(r)}{\psi}\,h(v|_{[0,s]}),\Real\duality{\im  B_m  v(r)}{\varphi}\Big)_{L^2(\tilde{\Omega}\times [s,t]\times \N)}
\nonumber\\&=\Etilde \bigg[\sumM \int_s^t  \Real\duality{\im  B_m  v(r)}{\psi} \Real\duality{\im  B_m  v(r)}{\varphi} \df r\,  h(v|_{[0,s]})\bigg]
\end{align}
for all $t\in [s,T]$, $\psi,\varphi\in \EA$.
As a combination of \eqref{martingaleLimitOne}, \eqref{martingaleLimitTwo}, \eqref{martingaleLimitThree}, and the embedding $U\hookrightarrow \EA$ we deduce 
	\begin{align}\label{approxMartingaleExpec}
	\Etilde \left[ \Real\duality{N(t)-N(s)}{\psi} h(v|_{[0,s]})\right]=0
	\end{align}
	and
	\begin{align}\label{approxMartingaleVariation}
	\Etilde \Bigg[& \Bigg(\Real\duality{N(t)}{\psi}\Real\duality{N(t)}{\varphi}-\Real\duality{N(s)}{\psi}\Real\duality{N(s)}{\varphi}\nonumber\\
	&\hspace{1 cm}-\sumM \int_s^t  \Real\duality{\im  B_m  v(r)}{\psi} \Real\duality{\im  B_m  v(r)}{\varphi} \df r\Bigg) h(v|_{[0,s]})\Bigg]=0
	\end{align}
	for all $0\le s\le t\le T$, $\psi, \varphi \in U$ and bounded, continuous functions $h$ on $C([0,T],U^*).$ 
			Now, let $\eta,\zeta \in{H}.$ Then $\iota^* \eta,\iota^* \zeta \in U$ and for all $z\in U^*,$ we have $\skpHReal{L z}{\eta}=\dualityReal{z}{\iota^* \eta}_{U^*,U}.$ By the previous considerations, $L N$ is a continuous, square integrable process in ${H}$ and the identities $\eqref{approxMartingaleExpec}$ and $\eqref{approxMartingaleVariation}$ imply
			\begin{align*}
			\Etilde \left[\skpHReal{L N(t)-L N(s)}{\eta} h(v|_{[0,s]})\right]=0
			\end{align*}
			and
			\begin{align*}
			\Etilde \Bigg[ \Bigg(\skpHReal{L N(t)}{\eta}&\skpHReal{L N(t)}{\zeta}-\skpHReal{L N(s)}{\eta}\skpHReal{L N(s)}{\zeta}\\
			&+\sumM \int_s^t  \skpHReal{\im L  B_m v(r)}{\eta} \skpHReal{\im L B_m v(r)}{\zeta} \df r\Bigg) h(v|_{[0,s]})\Bigg]=0.
			\end{align*}
			Lemma A.16 in \cite{dissertationFH} finally ensures that $L N$ is a continuous, square integrable martingale in ${H}$ with respect to the filtration $\tilde{\F}_{t}:=\sigma \left(v(s): s\le t\right)$ which has the quadratic variation						
			\begin{align*}
			\quadVar{L N}_t\zeta= \sumM \int_0^t \im L B_m v(s)\, \skpHReal{\im L B_m v(s)}{\zeta} \df s,\qquad \zeta \in {H}.
			\end{align*}
\end{proof}

 Combining Lemma~\ref{martingaleRepresentation} and the Martingale Representation Theorem from \cite{daPrato}, Theorem 8.2, we obtain the following result. 

\begin{Prop}\label{Prop:pseudoEquation}
	There exists a stochastic basis $( \tilde{\tilde{\Omega}},  \tilde{\tilde{\F}},  \tilde{\tilde{\Prob}}, (\tilde{\tilde{\F}}_t)_t)$ and a $Y$-cylindrical Wiener process $W'$ defined on 
	$\left(\Omega',\F',\Prob'\right)=\left(\tilde{\Omega} \times \tilde{\tilde{\Omega}}, \tilde{\F}\otimes \tilde{\tilde{\F}},  \tilde{\Prob}\otimes\tilde{\tilde{\Prob}}\right)$, adapted to 
	$\F'_t=\tilde{\F}_t\otimes \tilde{\tilde{\F}}_t$, $t\in [0,T]$ which $\Prob'$-almost surely for all $t\in [0,T]$ satisfies 
	\begin{align}\label{equationWithPseudoNonlin}
	v(t)= u_0+ \int_0^t \left[-\im A v(s)-\im  \tilde{F}(s)+\mu(v(s))\right] \df s-\im \int_0^t B(v(s))\df W'(s), 
	\end{align}
	where, with a slight abuse of notation,  we write
	\begin{align*}
		v(t,\omega')\colon=v(t,\tilde{\omega})\qquad \text{and}\qquad \tilde{F}(t,\omega')\colon=\tilde{F}(t,\tilde{\omega})
	\end{align*}
for $t\in [0,T]$,  $\omega'=(\tilde{\omega},\tilde{\tilde{\omega}})\in \Omega'$.
\end{Prop}

We skip the proof which is similar to the proof of Theorem 1.1 on page 49  in \cite{ExistencePaper}.
We would like to point out that at first glance, the equation \eqref{equationWithPseudoNonlin} should be understood in the auxiliary space $U^*$. 
By $(v,\tilde{F})\in Z$, however, each term is sufficiently regular that the equation holds in the natural space $\EAdual$. Of course, Proposition~\ref{Prop:pseudoEquation} is a crucial step towards our goal of proving the existence of a martingale solution to \eqref{ProblemStratonovich}. It remains to show $\tilde{F}(t)=F(v(t))$ almost surely for all $t\in [0,T]$. Before we tackle this problem, let us use \eqref{equationWithPseudoNonlin} to verify that the mass of $v$ is conserved. The following Lemmata are needed to be prepared for this proof.

\begin{Lemma}\label{RegularityOfSomeFunctionals}
	\begin{itemize}
		\item[a)] For all $n\in\N$, $t\in[0,T]$, the map $\Psi_{n,t}\colon Z\to \R$ given by 
		\begin{align}
		\Psi_{n,t}(u,G)=\Real \duality{\mathbf{1}_{M_n}G(t)}{\im u(t)}_{L^{(\alpha+1)/\alpha},L^{\alpha+1}}
		\end{align}	
		is continuous.
		\item[b)] For all $n\in\N$, $t\in[0,T]$, $g\in L^{\alpha+1}(M)$, the map $\Phi_{n,t,g}\colon Z\to \R$ given by 
		\begin{align*}
		\Phi_{n,t,g}(u,G)=\duality{\mathbf{1}_{M_n}(G(t)-F(u(t)))}{g}_{L^{(\alpha+1)/\alpha},L^{\alpha+1}}
		\end{align*}	
		is continuous.		 
		\item[c)] For all $t\in [0,T]$, the map $\Phi_t\colon Z\to \R$ given by 
		\begin{align*}
		\Phi_t(u,G)=\norm{G(t)-F(u(t))}_{L^{(\alpha+1)/\alpha}}
		\end{align*}
		is Borel-measurable.
	\end{itemize}
\end{Lemma}

\begin{proof}
	Let us fix $n\in\N$, $t\in[0,T]$  and choose $(u_k,G_k), (u,G)\in Z$ such that $(u_k,G_k)\to (u,G)$ in $Z$ as $k\to\infty$. This and Assumption \ref{spaceAssumptions} imply
	\begin{align}\label{RegularityOfSomeFunctionals:Convergences}
	u_k(t)\rightharpoonup u(t)\quad \text{in $\EA$},\qquad u_k(t)\to u(t)\quad  \text{in $L^{\alpha+1}(M_n)$},\qquad G_k(t)\rightharpoonup G(t) \quad \text{in $L^{(\alpha+1)/\alpha}(M)$.}
	\end{align}
	In particular, the sequence $(G_k(t))_{k\in\N}\subset L^{(\alpha+1)/\alpha}(M)$ is bounded and we get 
	\begin{align*}
	&\vert\Psi_{n,t}(u,G)-\Psi_{n,t}(u_k,G_k)\vert=\vert\Real \duality{\mathbf{1}_{M_n}G(t)}{\im u(t)}-\Real \duality{\mathbf{1}_{M_n}G_k(t)}{\im u_k(t)}\vert
	\\&\le \vert\Real \duality{\mathbf{1}_{M_n}[G(t)-G_k(t)]}{\im u(t)}\vert
	+\vert\Real \duality{G_k(t)}{\im \mathbf{1}_{M_n} [u(t)- u_k(t)]}\vert
	\\&\le \vert\Real \duality{\mathbf{1}_{M_n}[G(t)-G_k(t)]}{\im u(t)}\vert+\sup_{k\in\N}\norm{G_k(t)}_{L^{(\alpha+1)/\alpha}}\, \norm{\mathbf{1}_{M_n} [u(t)- u_k(t)]}_{\LalphaPlusEinsKurz}
	\\&\longrightarrow 0\qquad \text{for $k\to \infty$},
	\end{align*}
	which shows item~a).
	Next note that \eqref{RegularityOfSomeFunctionals:Convergences} and the Lipschitz property of $F$ show $F(u_k(t))\to F(u(t))$ in $L^{(\alpha+1)/\alpha}(M_n)$. 
	As a consequence, we obtain 
	\begin{align*}
	\vert\Phi_{n,t,g}(u,G)-\Phi_{n,t,g}(u_k,G_k)\vert&\le \vert  \duality{\mathbf{1}_{M_n}[G(t)-G_k(t)]}{g}\vert +\vert  \duality{\mathbf{1}_{M_n}[F_k(u(t))-F(u(t))]}{g}\vert
	\\&\longrightarrow 0\qquad \text{for $k\to \infty$},
	\end{align*}
	for all $g\in L^{\alpha+1}(M)$. This is sufficient for item~b).
	Since $L^{(\alpha+1)/\alpha}(M)$ is separable with dual space  $(L^{(\alpha+1)/\alpha}(M))^*=L^{\alpha+1}(M)$, we can take
	$(g_l)_{l\in\N}\subset \LalphaPlusEins$ such that for all $f\in L^{(\alpha+1)/\alpha}(M)$, we obtain
	\begin{align*}
	\norm{f}_{L^{(\alpha+1)/\alpha}}=\sup_{l\in\N}\vert\duality{f}{g_l}_{L^{(\alpha+1)/\alpha},L^{\alpha+1}}\vert.
	\end{align*}
	Lebesgue's dominated convergence theorem hence proves
	\begin{align*}
	\Phi_t(u,G)&=\norm{G(t)-F(u(t))}_{L^{(\alpha+1)/\alpha}}=\lim_{n\to\infty} \norm{\mathbf{1}_{M_n}[G(t)-F(u(t))]}_{L^{(\alpha+1)/\alpha}}
	\\&=\lim_{n\to\infty} \sup_{l\in\N} \vert\duality{\mathbf{1}_{M_n}(G(t)-F(u(t)))}{g_l}_{L^{(\alpha+1)/\alpha},L^{\alpha+1}}\vert=\lim_{n\to\infty} \sup_{l\in\N} \vert\Phi_{n,t,g_l}(u,G)\vert.
	\end{align*}
	Combining this with item~b) establishes item~c).
\end{proof}

\begin{Lemma}\label{CancellationTildeF}
	Assume the setting of Proposition~\ref{NewAprioriBounds}. Then
	\begin{itemize}
		\item[a)]  it holds for all $k\in\N$ that $F(v_k(t))=\tilde{F}_k(t)$ almost surely for all $t\in [0,T]$ and
		\item[b)] it holds almost surely for all $t\in [0,T]$ that
		\begin{align*}
		\Real \duality{\tilde{F}(t)}{\im v(t)}_{L^{(\alpha+1)/\alpha},\LalphaPlusEinsKurz}=0.
		\end{align*}		
	\end{itemize}
\end{Lemma}

\begin{proof}
	\emph{Step 1.} By item c) in Lemma \ref{RegularityOfSomeFunctionals} and Proposition~\ref{PropAlmostSureConvergence}, we obtain for all $t\in [0,T]$   that 
	\begin{align*}
	\Etilde\norm{\tilde{F}_k(t)-F(v_k(t))}_{\LalphaPlusEinsDualNorm}&= \int_Z \norm{G(t)-F(u(t))}_{\LalphaPlusEinsDualNorm} \,\tilde{\Prob}^{(v_k,\tilde{F}_k)}(\df u,\df G)\\
	&= \int_Z \norm{G(t)-F(u(t))}_{\LalphaPlusEinsDualNorm} \,\Prob^{(u_{n_k},F(u_{n_k}))}(\df u,\df G)
	\\&=\E\norm{F(u_{n_k}(t))-F(u_{n_k}(t))}_{\LalphaPlusEinsDualNorm}=0.
	\end{align*}
	We infer that $\tilde{F}(t)=F(v_k(t))$ for all $t\in [0,T]$ almost surely. Moreover, observe that Assumption~\ref{spaceAssumptions}~v) and  Assumption~\ref{nonlinearAssumptions}~ii) imply that the map  $u\ni  X_\gamma\mapsto F(u)\in \LalphaPlusEinsDual$ is continuous. The fact that $(v_k,\tilde{F}_k)\in Z$ therefore yields  $\tilde{F}_k, F(v_k)\in C([0,T];\LalphaPlusEinsDual)$ almost surely.  We infer that
	$\tilde{F}(t)=F(v_k(t))$ almost surely  for all $t\in [0,T]$, which establishes item~a). \\
	\emph{Step 2.} From the first step and Assumption~\ref{nonlinearAssumptions}~i) we conclude that
	almost surely  for all $n\in\N$, $t\in [0,T]$  we have
	\begin{align*}
	\Real \duality{\mathbf{1}_{M_n}\tilde{F}_k(t)}{\im v_k(t)}_{L^{(\alpha+1)/\alpha},\LalphaPlusEinsKurz}
	=\Real \duality{\mathbf{1}_{M_n}F(v_k(t))}{\im v_k(t)}_{L^{(\alpha+1)/\alpha},\LalphaPlusEinsKurz}=0.
	\end{align*}
	Let us recall from Proposition~\ref{PropAlmostSureConvergence} that  $(v_k,\tilde{F}_k)\to (v,\tilde{F})$ almost surely  in $Z$.
	Item a) in Lemma~\ref{RegularityOfSomeFunctionals} therefore implies  almost surely for all $n\in\N$,  $t\in[0,T]$ that 
	\begin{align*}
	\Real \duality{\mathbf{1}_{M_n}\tilde{F}(t)}{\im v(t)}_{L^{(\alpha+1)/\alpha},\LalphaPlusEinsKurz}
	=\lim_{k\to\infty}\Real \duality{\mathbf{1}_{M_n}\tilde{F}_k(t)}{\im v_k(t)}_{L^{(\alpha+1)/\alpha},\LalphaPlusEinsKurz}=0.
	\end{align*}
	Then, Lebesgue's dominated convergence theorem proves almost surely for all $t\in[0,T]$ that 
	\begin{align*}
	\Real \duality{\tilde{F}(t)}{\im v(t)}_{L^{(\alpha+1)/\alpha},\LalphaPlusEinsKurz}
	=\lim_{n\to \infty}\Real \duality{\mathbf{1}_{M_n}\tilde{F}(t)}{\im v(t)}_{L^{(\alpha+1)/\alpha},\LalphaPlusEinsKurz}=0.
	\end{align*}
	The proof of item~b) is thus completed.
\end{proof}

\begin{Lemma}\label{massConservationMartingale}
	The following assertions hold: 
	\begin{enumerate}[a)]
		\item We have $\norm{v(t)}_H=\norm{u_0}_H$ almost surely for all $t\in [0,T]$.
		\item We have $\norm{v_k(t)}_H=\norm{S_{n_k}u_0}_H$ almost surely for all $t\in [0,T]$, $k\in\N$.
	\end{enumerate}
\end{Lemma}


\begin{proof}
	\emph{Step 1.} 
	We formally apply  Ito's formula to the It\^o process from \eqref{equationWithPseudoNonlin} and the function
	$\mathcal{M}: {H} \to \R$ defined by $\mathcal{M}(w):=\norm{w}_{H}^2,$ which is twice continuously Fr\'{e}chet-differentiable with 
	\begin{align*}
	\mathcal{M}'[w]h_1&= 2 \Real \skpH{w}{ h_1}, \qquad
	\mathcal{M}''[w] \left[h_1,h_2\right]= 2 \Real \skpH{ h_1}{h_2}
	\end{align*}
	for $w, h_1, h_2\in {H}.$ This yields
	\begin{align}\label{Formalcalculation}
	\norm{ v(t)}_{H}^2=&\norm{ u_0}_{H}^2+2 \int_0^t \Real \skpH{ v(s)}{-\im A v(s)-\im  \tilde{F}(s)+\mu(v(s))} \df s\nonumber\\
	&- 2  \int_0^t \Real \skpH{ v(s)}{\im B(v(s))\df W'(s) }
	+\sumM \int_0^t   
	\Vert  B_m v(s)\Vert_{H}^2\df s
	\end{align}
	almost surely  in $[0,T]$. 	
	In addition, observe that the assumption that $B_m$, $m\in\N$, are selfadjoint operators implies $\Real \skpH{ v(t)}{\im  B_m v(t)}=0$ and
	\begin{align*}
	\sumM \int_0^t \Vert  B_m v(s)\Vert_{H}^2\df s=\sumM\int_0^t\Real \skpH{ v(s)}{B_m^2v(s)} \df s=-2\int_0^t\Real \skpH{ v(s)}{\mu(v(s))} \df s. 
	\end{align*}
	This, \eqref{Formalcalculation}, and the formal identity $\Real \skpH{ v(t)}{\im  A v(t)}=0$
	ensure
		\begin{align}\label{FormalcalculationTwo}
		\norm{ v(t)}_{H}^2=\norm{ u_0}_{H}^2+2 \int_0^t \Real \skpH{ v(s)}{-\im  \tilde{F}(s)} \df s
		\end{align}
		almost surely  in $[0,T]$.	Since Lemma~\ref{CancellationTildeF} implies $\Real \duality{\tilde{F}(t)}{\im v(t)}=0$ almost surely for all $t\in [0,T]$, this is sufficient to formally obtain the assertion in item~a). 
	The calculation from above can be made rigorous by a regularization procedure via Yosida approximations $\Yosida:=\nu\left(\nu+A\right)^{-1}$ for $\nu>0$
	applied to \eqref{equationWithPseudoNonlin}
	 and a limit process $\nu\to \infty$. This is standard and we refer to, e.g., \cite{BarbuL2}, \cite{ExistencePaper}, and \cite{dissertationFH} for are more detailed treatment of this procedure. \\
	
	\emph{Step 2.} From Lemma~\ref{MassEstimateGalerkinSolution}, we already know that $\norm{u_n(t)}_{H}=\norm{S_n u_0}_{H}$ almost surely for all $n\in\N$, $t\in [0,T]$. 
	Item~b) then follows immediately since the laws of $u_{n_k}$ and $v_k$ coincide by Proposition~\ref{PropAlmostSureConvergence}.
\end{proof}

\begin{Korollar}\label{Cor:NonlinearityCoincides}
	We have $\tilde{F}(t)=F(v(t))$ $\Prob'$-almost surely for all $t\in [0,T]$.
\end{Korollar}

\begin{proof}
	Note that Lemma~\ref{convergenceProperty}  and Lemma~\ref{massConservationMartingale} prove that $\norm{v_k(t)}_H\to \norm{v(t)}_H$ almost surely for all $t\in [0,T]$.
	Moreover, Proposition~\ref{PropAlmostSureConvergence} shows  $v_k(t)\rightharpoonup v(t)$ in $H$ almost surely for all $t\in [0,T]$. We therefore deduce $v_k(t)\to v(t)$ in $H$ almost surely for all $t\in [0,T]$.
	 Lemma \ref{CancellationTildeF}, the embedding $X_\gamma\hookrightarrow \LalphaPlusEins$, interpolation of $X_\gamma$ between $H$ and $\EA$, and the fact that the sequence $(v_k(t))_k$ in $\EA$ is bounded almost surely for all $t\in [0,T]$ due to weak convergence prove 
	\begin{align*}
		\norm{\tilde{F}_k(t)-F(v(t))}_{L^{\alpha/(\alpha+1)}}&=\norm{F(v_k(t))-F(v(t))}_{L^{\alpha/(\alpha+1)}}
		\\&\lesssim \big(\norm{v_k(t)}_{L^{\alpha+1}}+\norm{v(t)}_{L^{\alpha+1}}\big)^{\alpha-1} \norm{v_k(t)-v(t)}_{L^{\alpha+1}}	
				\\&\lesssim \big(\norm{v_k(t)}_{X_\gamma}+\norm{v(t)}_{X_\gamma}\big)^{\alpha-1} \norm{v_k(t)-v(t)}_{X_\gamma}
		\\&\lesssim \big(\norm{v_k(t)}_{\EA}+\norm{v(t)}_{\EA}\big)^{\alpha-1} \norm{v_k(t)-v(t)}_{\EA}^{2\gamma} \norm{v_k(t)-v(t)}_{H}^{1-2\gamma}
				\\&\lesssim \big(\sup_{k\in\N}\norm{v_k(t)}_{\EA}+\norm{v(t)}_{\EA}\big)^{\alpha-1+2\gamma}  \norm{v_k(t)-v(t)}_{H}^{1-2\gamma}
		\to 0 
	\end{align*}
	almost surely for all $t\in[0,T].$ Proposition \ref{PropAlmostSureConvergence} shows $\tilde{F}_k(t)\rightharpoonup \tilde{F}(t)$ in $\LalphaPlusEinsDual$ almost surely for all $t\in[0,T].$
	We thus have $\tilde{F}(t)=F(v(t))$ almost surely for all $t\in [0,T]$.
\end{proof}

\begin{Theorem}\label{MainTheorem}
	Let $T>0$ and $u_0\in E_A.$  Under the Assumptions $\ref{spaceAssumptions}$, $\ref{nonlinearAssumptions}$, $\ref{focusing},$ $\ref{stochasticAssumptions},$ there exists an analytically weak global martingale solution  $\left({\Omega}',{\F}',{\Prob}',{W}',{\Filtration}',v\right)$ of $\eqref{ProblemStratonovich}$ which additionally satisfies $v\in \weaklyContinousEA$ almost surely, $v\in L^q({\Omega}',\LinftyEA)$ 
	for all $q\in [1,\infty)$, and
$\norm{v(t)}_{H}=\norm{u_0}_{H}$
		almost surely for all $t\in [0,T].$
\end{Theorem}

\begin{proof}
	We choose the system $({\Omega}',{\F}',{\Prob}',{W}',{\Filtration}',v)$ as in Proposition~\ref{Prop:pseudoEquation} and recall that 
	for $t\in [0,T]$,  $\omega'=(\tilde{\omega},\tilde{\tilde{\omega}})\in \Omega'$, we have 
		\begin{align*}
		v(t,\omega'):=v(t,\tilde{\omega})\qquad \text{and}\qquad \tilde{F}(t,\omega'):=\tilde{F}(t,\tilde{\omega}).
		\end{align*}
	Combining this and Corollary~\ref{Cor:NonlinearityCoincides} with the fact that $v\colon\tilde{\Omega}\to Z$ and
	\begin{align*}
		Z= C_w([0,T],\EA\times \LalphaPlusEinsDual)\cap C([0,T],X_\gamma\times \LalphaPlusEinsDual),
	\end{align*}
	we obtain 	
	\begin{align*}
	v(t)= u_0+ \int_0^t \left[-\im A v(s)-\im  F(v(s))+\mu(v(s))\right] \df s-\im \int_0^t B(v(s))\df W'(s) 
	\end{align*}
	almost surely for all $t\in [0,T]$, $u\in \weaklyContinousEA$ almost surely, $v\in L^q({\Omega}',\LinftyEA)$ 
		for all $q\in [1,\infty)$, and
		$\norm{v(t)}_{H}=\norm{u_0}_{H}$
		almost surely for all $t\in [0,T].$ It therefore remains to prove that $u$ has almost surely continuous paths in $X_\theta$ for $\theta\in (\gamma,1/2)$. We can deduce this from $v\in C([0,T],X_\gamma)$, $v\in L^q({\Omega}',\LinftyEA)$, and the interpolation inequality
		\begin{align*}
			\norm{w}_{X_\theta}\lesssim \norm{w}_{X_\gamma}^{(2\theta-1)/(2\gamma-1)}\norm{w}_{\EA}^{(2\gamma-2\theta)/(2\gamma-1)}.
		\end{align*} 
\end{proof}

%

%% file: Examples.tex
\section{Examples}\label{ExamplesSection}

In this section, we would like to use the existence result for the abstract stochastic NLS proved in Theorem~\ref{MainTheorem} to deduce Theorem~\ref{mainTheoremIntro} which deals with the concrete cases of domains and manifolds. We would like to point out that the choice of examples is rather illustrative to keep the level of technicality low and there are other situation in which Theorem~\ref{MainTheorem} applies. 
As in \cite{ExistencePaper} and \cite{dissertationFH} it is possible to replace the Laplacian $-\Delta$ by the fractional Laplacian $(-\Delta)^\beta$ if one adapts the admissible range for $\alpha$ according to the choice of the fractional power $\beta$. Another possible generalization can be made by replacing the Laplacians by general elliptic operators with essentially bounded real valued coefficients (cf., e.g. \cite{OuhabazHeatEquations}, Chapter 6, for Gaussian estimates in this case). Moreover, we believe that our setting can be also used in the context of Schrödinger operators $A=-\Delta+V$ on $\Rd$ for some potentials $V$ subject to further conditions and in the context of Laplacians on graphs or fractals. 

We point out that in the setting of Theorem~\ref{mainTheoremIntro}, it is straightforward and similar to \cite{ExistencePaper} and \cite{dissertationFH} to verify that Assumptions \ref{nonlinearAssumptions}, \ref{focusing}, and \ref{stochasticAssumptions} are indeed true. In the following, we therefore concentrate on Assumption~\ref{spaceAssumptions}.

\subsection{The stochastic NLS on domains in $\R^d$}

Let $M\subset \Rd$ be a domain. We define fractional Sobolev spaces $H^s(M)$, $s\ge 0$, via complex interpolation, i.e. $H^s(M)=W^{s,2}(M)$ for $s\in \N_0$ and 
\begin{align*}
H^s(M)=\big[W^{ \lfloor s\rfloor,2}(M),W^{\lfloor s\rfloor+1,2}(M)\big]_{s-\lfloor s\rfloor},\qquad s\in [0,\infty)\setminus \N_0
\end{align*}
and, similarly, $H^s_0(M)=W^{s,2}_0(M)$ for $s\in \N_0$ and 
\begin{align*}
H^s_0(M)=\big[W^{ \lfloor s\rfloor,2}_0(M),W^{\lfloor s\rfloor+1,2}_0(M)\big]_{s-\lfloor s\rfloor},\qquad s\in [0,\infty)\setminus \N_0.
\end{align*}
For $V\in \{H^1_0(M),H^1(M)\}$, we further consider the form 
$\mathfrak{a}_V: V \times V \to \C$ with
\begin{align*}
\mathfrak{a}_V(u,v)=\int_M \nabla u \cdot \nabla \bar{v} \, \df x, \quad u,v\in V.
\end{align*}
In the case $V=H^1_0(M)$, the operator associated to $\mathfrak{a}_V$ is the  Dirichlet Laplacian $-\Delta_D$.  
The Neumann Laplacian $-\Delta_N$ is associated to $\mathfrak{a}_{H^1(M)}$. Both $-\Delta_D$ and $-\Delta_N$ are selfadjoint and nonnegative.

\subsubsection{Neumann boundary conditions}

From now on, we assume that $M\subset \Rd$ is a  Lipschitz domain. 
Note that due to the square root property (cf.,e.g., \cite{OuhabazHeatEquations}, Theorem 8.1), we have $E_A=H^1(M).$ By choosing $S=\varepsilon-\Delta_N$ for some $\varepsilon>0$ and applying Theorem 6.10 in \cite{OuhabazHeatEquations}, we guarantee that $e^{-tS}$ has Gaussian bounds. This proves Assumption~\ref{spaceAssumptions} (i) for $p_0=1$. 
Let us recall Stein's extension theorem (cf., e.g., Adams \& Fournier \cite{Adams}, Theorem 5.24): There is a total extension operator $E_M$ for $M$, i.e. a linear operator $E_M\colon L^0(M)\to L^0(\R^d)$ such that for every $p\in [1,\infty)$, $k\in\N_0$ there is $K=K(k,p)\in (0,\infty)$ such that for every $u\in W^{k,p}(M)$ it holds that
\begin{align}\label{Stein}
	E_M u=u\qquad \text{a.e. on $M$} \qquad \text{and}\qquad \norm{E_M u}_{W^{k,p}(\R^d)}\le K \norm{u}_{W^{k,p}(M)}.
\end{align}
Note that by complex interpolation, this proves $E_M\in \LinearOperatorsTwo{H^s(M)}{H^s(\R^d)}$.
Since $A$ is a selfadjoint operator, we have
\begin{align*}
	X_{\theta/2}=[H,\EA]_{\theta}=[L^2(M),H^1(M)]_{\theta}=H^\theta(M),\qquad 0\le \theta \le 1.
\end{align*}
Due to the assumption $\alpha \in (1,1+\frac{4}{(d-2)_+})$ and, e.g, Theorem 2.8.1 in \cite{TriebelInterpolationTheory},  there is $\gamma\in [0,1/2)$ such that the Sobolev embedding $H^{2\gamma}(\R^d)\hookrightarrow L^{\alpha+1}(\Rd)$ holds. This and \eqref{Stein} yield that  there is $\gamma\in [0,1/2)$ with $X_\gamma\hookrightarrow \LalphaPlusEins$, since
\begin{align*}
	\norm{u}_{\LalphaPlusEins}=\norm{E_M u}_{\LalphaPlusEins}\le \norm{E_M u}_{L^{\alpha+1}(\Rd)}\lesssim \norm{E_M u}_{H^{2\gamma}(\R^d)}\lesssim \norm{ u}_{H^{2\gamma}(M)}=\norm{ u}_{X_\gamma}.
\end{align*}
We have thus established Assumption~\ref{spaceAssumptions} (iii). Finally, Assumption~\ref{spaceAssumptions} (ii) is a consequence of the following Lemma and $E_A=H^1(M).$

\begin{Lemma}
	Let $M\subset \R^d$ be a Lipschitz domain, let $M_n\subset M$, $n\in\N$, satisfy $M_n=\{x\in M\colon \vert x\vert\le n\}$, and take  $(u_k)_{k\in\N}\subset H^1(M)$, $u\in H^1(M)$ with $u_k\rightharpoonup u$ in $H^1(M)$ for $k\to \infty$. Then, for all $n\in\N$ it holds that $u_k\to u$ in $L^{\alpha+1}(M_n)$ as $k\to \infty$.
\end{Lemma}

\begin{proof}
	We fix $n\in\N$ and observe that the restriction operator $\mathfrak{R}\colon H^1(\Rd)\to H^1(M)$, $\mathfrak{R}u=u|_M$ is bounded. 
	Hence, we deduce $\mathfrak{R}E_Mu_k\rightharpoonup \mathfrak{R}E_Mu$ in $H^1(B(0,n))$ and from the Rellich-Kondrachov theorem, we infer  that $H^1(B(0,n))\hookrightarrow L^{\alpha+1}(B(0,n))$ is compact. The fact that compact operators map weakly convergent sequences to strongly convergent ones thus proves $\mathfrak{R}E_Mu_k\to \mathfrak{R}E_Mu$ in $L^{\alpha+1}(B(0,n))$ and therefore, $u_k\to u$ in $L^{\alpha+1}(M_n)$.
\end{proof}

\subsubsection{Dirichlet boundary conditions}

Now, let $M\subset \R^d$ be a domain and $A=-\Delta_D$.
As in the Neumann case, we get $\EA=H^1_0(M)$ as a consequence of the square root property and with $S=\varepsilon-\Delta_D$ for an arbitrary $\varepsilon>0$, Theorem 6.10 in \cite{OuhabazHeatEquations} shows that $e^{-tS}$ has Gaussian bounds. This proves Assumption~\ref{spaceAssumptions} (i) for $p_0=1$. 
Moreover, we observe that the trivial extension operator $E_{M,0}\colon L^0(M)\to L^0(\R^d)$ which extends each function on $M$ by $0$ to $\Rd$ satisfies for every $p\in [1,\infty)$, $k\in\N_0$, $u\in W^{k,p}_0(M)$ that
\begin{align}\label{trivialExtension}
E_{M,0} u=u\qquad \text{a.e. on $M$} \qquad \text{and}\qquad \norm{E_{M,0} u}_{W^{k,p}_0(\R^d)}= \norm{u}_{W^{k,p}_0(M)}.
\end{align}
Riesz-Thorin hence proves $E_{M,0}\in \LinearOperatorsTwo{H^s_0(M)}{H^s(\R^d)}$.
Using the same arguments as in the Neumann case with $E_M$ replaced by $E_{M,0}$ it is now straightforward to complete the proof of Assumption~\ref{spaceAssumptions} for the Dirichlet Laplacian.

\subsection{The stochastic NLS on the full space}

In contrast to the case of domains in $\Rd$, it is actually simpler to verify Assumption~\ref{spaceAssumptions} for $M=\Rd$ and $A=-\Delta$ with $S=\varepsilon-\Delta$. Item~i) holds since the heat kernel equals the density of the Gaussian distribution of $\Rd$, item~ii) is true for $M_n=B(0,n)$, $n\in\N$, and item~iii) follows from the Sobolev embeddings of the Bessel potential spaces on $\Rd$. Actually, our procedure is not the natural one to deal with $M=\Rd$ as we neither use Strichartz estimates nor other consequences of the dispersive behavior of the Schr\"odinger group on $\Rd$. In the deterministic theory, these methods have been proved very useful and instead of going into further details, we just refer to the popular monographs \cite{Tao}, \cite{Cazenave}, and \cite{Linares}. Also in the presence of stochastic noise, Strichartz estimates have been succesfully applied in the literature. We would like to mention articles using a truncation of the nonlinearity and estimates for the stochastic convolution (e.g. \cite{BouardLzwei}, \cite{BouardHeins} and \cite{FHornung}, \cite{dissertationFH}) and articles which employ the rescaling approach  (e.g. \cite{BarbuL2}, \cite{BarbuH1}, \cite{BarbuNoBlowUp}, \cite{HerrScattering}, \cite{zhangCritical}, \cite{zhang2017}). All these results are stronger compared to Theorem~\ref{mainTheoremIntro} in the sense that they prove existence \& uniqueness of stochastically strong solutions. To the best of our knowledge, however, even in the $\Rd$-setting, Theorem~\ref{mainTheoremIntro} contains new results. On the one hand,   the regularity and decay assumptions on the noise coefficients in the papers with the rescaling approach are much stronger than \eqref{MultiplierClassIntro}. On the other hand, the results in \cite{BouardHeins} and \cite{dissertationFH} for $H^1$ initial data do not cover the full scale of subcritical exponents $\alpha$ as in i) and ii) of Theorem~\ref{mainTheoremIntro}.  

\subsection{The stochastic NLS on manifolds}

Let $(M,g)$ be a Riemannian manifold such that

\begin{align}\label{ManifoldAssumption}
\text{$M$ is complete, has a positive injectivity radius, bounded geometry,}\nonumber\\ \text{and nonnegative Ricci curvature.}
\end{align}

We refer to \cite{Triebel}, Chapter 7, for a definition of the notions above and remark that compact manifolds with nonnegative Ricci curvature satisfy \eqref{ManifoldAssumption}.
We emphasize that in the compact case, the doubling property is also true without curvature bound and thus, the assumption of nonnegative Ricci curvature can be dropped. In this sense, the result of the present paper in the manifold-setting generalizes the existence result from \cite{ExistencePaper}. 
Let $X=M$, $\rho$ be the geodesic distance and $\mu$ be the canonical volume measure on $X$.
 Moreover, let $A:=-\Delta_g$, where $\Delta_g$ denotes the Laplace-Beltrami operator on $M$. The operator $A$ is associated to the form  $\mathfrak{a}: H^1(M) \times H^1(M) \to \C$ with
\begin{align*}
\mathfrak{a}(u,v)=\int_M \nabla u \cdot \nabla \bar{v} \, \df \mu(x), \quad u,v\in H^1(M)
\end{align*}
and thus, the square root property shows $\EA=H^1(M)$. Let $S:=I-\Delta_g.$ Then, $S$ is strictly positive. For a comprehensive study of Gaussian estimates on manifolds, we refer to \cite{GrigoryanLectureNotes}. 
In short, we note that \eqref{ManifoldAssumption}  implies that $M$ admits a relative Faber-Krahn inequality (cf. \cite{GrigoryanLectureNotes}, Definition 6.12). Thus, Theorem 6.13   in \cite{GrigoryanLectureNotes} (see also Corollary 6.14 in \cite{GrigoryanLectureNotes} for an untechnical overview of the heat kernel estimates on manifolds) implies the doubling property \eqref{doubling} and the upper Gaussian estimate for $e^{-tS}$, i.e. \eqref{generalizedGaussianEstimate} for $p_0=1$.  
Furthermore, we refer to, e.g., Theorem III.1.2 in \cite{Bolleyer} for a proof of Assumption~\ref{spaceAssumptions} ii) in the present setting.
As in the case of domains, one can verify Assumption~\ref{spaceAssumptions} iii) by complex interpolation and embedding results of fractional Sobolev space. We refer to \cite{Bolleyer}, Theorem II.1.2 and \cite{Triebel}, chapter 7, for these properties.